\documentclass[11pt, twoside, letterpaper]{amsart}
\usepackage{amsmath}
\usepackage{hyperref}
\usepackage{amsmath}
\usepackage{color}
\usepackage{amssymb}
\usepackage{enumerate}
\usepackage{epsfig}

\newtheorem{thm}{Theorem}[section]

\newtheorem{lem}[thm]{Lemma}
\newtheorem{prop}[thm]{Proposition}
\theoremstyle{definition}

\newtheorem{ass}[thm]{Assumption}

\theoremstyle{remark}
\newtheorem{rem}[thm]{Remark}
\newtheorem{exa}[thm]{Example}
\numberwithin{equation}{section}
\newcommand{\abs}[1]{\left\vert#1\right\vert}
\newcommand{\set}[1]{\left\{#1\right\}}
\newcommand{\Real}{\mathbb R}
\newcommand{\spos}{\mathbb S_{++}}
\newcommand{\Natural}{\mathbb N}
\newcommand{\eps}{\varepsilon}

\newcommand{\such}{\ | \ }
\newcommand{\nin}{n \in \Natural}
\newcommand{\tir}{t \in \Real_+}
\newcommand{\prob}{\mathbb{P}}

\newcommand{\qprob}{\mathbb{Q}}

\newcommand{\expec}{\mathbb{E}}

\newcommand{\basisp}{(\Omega, \, \bF, \, \prob)}

\newcommand{\F}{\mathcal{F}}
\newcommand{\G}{\mathcal{G}}

\newcommand{\ud}{\mathrm d}
\newcommand{\inner}[2]{\big \langle #1 , #2 \big \rangle}
\newcommand{\limn}{\lim_{n \to \infty}}

\newcommand{\B}{\mathcal{B}}

\newcommand{\cS}{{}^\mathsf{c} \kern-0.21em S}
\newcommand{\cH}{{}^\mathsf{c} \kern-0.23em H}
\newcommand{\cMM}{{}^\mathsf{c} \kern-0.19em [M, M]}

\newcommand{\pare}[1]{\left(#1\right)}
\newcommand{\bra}[1]{\left[#1\right]}
\newcommand{\cbra}[1]{\left\{#1\right\}}
\newcommand{\dbra}[1]{[\kern-0.15em[ #1 ]\kern-0.15em]}
\newcommand{\dbraco}[1]{[\kern-0.15em[ #1 [\kern-0.15em[}
\newcommand{\dbraoc}[1]{]\kern-0.15em] #1 ]\kern-0.15em]}
\newcommand{\dbraoo}[1]{]\kern-0.15em] #1 [\kern-0.15em[}

\newcommand{\bF}{\mathbf{F}}

\newcommand{\Lb}{\mathbb{L}}

\newcommand{\dfn}{\, := \,}

\newcommand{\bG}{\mathbf{G}}
\newcommand{\indic}{\mathbb{I}}

\newcommand{\ze}{\zeta}

\newcommand{\kin}{k \in \Natural}

\newcommand{\nada}[1]{}
\newcommand{\indep}{\perp\!\!\!\perp}

\newcommand{\tbF}{\widetilde{\bF}}

\newcommand{\tOmega}{\widetilde{\Omega}}

\newcommand{\pih}{\widehat{\pi}}

\newcommand{\dvg}[1]{\textrm{div}\left(#1\right)}
\newcommand{\dvgalt}[1]{\textrm{div}_{(z,x)}\left(#1\right)}

\begin{document}

\title{Continuous-Time Perpetuities and Time Reversal of Diffusions}%

\author{Constantinos Kardaras}%
\address{Constantinos Kardaras, Department of Statistics, London School of Economics, 10 Houghton Street, London, WC2A 2AE, England.}%
\email{k.kardaras@lse.ac.uk}%

\author{Scott Robertson}
\address{Scott Robertson, Department of Mathematical Sciences, Carnegie Mellon University, Wean Hall 6113,  Pittsburgh, PA 15213, USA.}%
\email{scottrob@andrew.cmu.edu}%

%\thanks{This work is supported in part by the National Science Foundation under grant number DMS-0908461.}
%\subjclass[2000]{60G07, 60G44}%

\date{\today}%

%----------------------------------------------------------------
\begin{abstract}
We consider the problem of estimating the joint distribution of a continuous-time perpetuity and the underlying factors which govern the cash flow rate, in an ergodic Markovian model. Two approaches are used to obtain the distribution. The first identifies a partial differential equation for the conditional cumulative distribution function of the perpetuity given the initial factor value, which under certain conditions ensures the existence of a density for the perpetuity. The second (and more general) approach, using techniques of time reversal, identifies the joint law as the stationary distribution of an ergodic multi-dimensional diffusion. This later approach allows for efficient use of Monte-Carlo simulation, as the distribution is obtained by sampling a single path of the reversed process.
\end{abstract}

\maketitle

% ----------------------------------------------------------------

\section*{Introduction} \label{sec: perpe}

\subsection*{Discussion}

In this article, we consider a \emph{continuous-time perpetuity} given by the random variable
\begin{equation}\label{eq: X_def}
X_0 \dfn \int_0^\infty D_t f(Z_t) \ud t.
\end{equation}
Above, $Z = \pare{Z_t}_{\tir}$ represents the value of an economic factor that determines a cash flow rate $\pare{f(Z_t)}_{\tir}$.  Cash flows are discounted according to $D = \pare{D_t}_{\tir}$; therefore, $X_0$ represents the whole payment in units of account at time zero. Our main concern is the identification of an efficient means to obtain the joint distribution of $(Z_0,X_0)$, as naive estimation of the distribution by simulating sample paths of $Z$ and approximating $X_0$ through numerical integration may be prohibitively slow. As $Z_0$ is typically observable, the joint distribution of $(Z_0,X_0)$ also allows us to obtain the conditional distribution of $X_0$ given $Z_0$.

In order to make the problem tractable, we work in a diffusive, Markovian environment where $Z$ and $D$ are solutions to the respective stochastic differential equations (written in integrated form)\footnote{Throughout the text, the prime symbol ($'$) denotes transposition.}
\begin{equation} \label{eq: Z_defn}
Z = Z_0 + \int_0^\cdot m(Z_t) \ud t + \int_0^\cdot \sigma(Z_t) \ud W_t,
\end{equation}
\begin{equation} \label{eq: D_defn}
D = 1 - \int_0^\cdot D_t \left(a(Z_t)\ud t + \theta(Z_t)' \sigma(Z_t)\ud W_t + \eta(Z_t)' \ud B_t \right).
\end{equation}
In the above equations, $W$ and $B$ are independent Brownian motions of dimension $d$ and $k$ respectively, while $m$, $\sigma$, $a$, $\theta$ and $\eta$ are given functions. (Precise assumptions on all the model coefficients are given in Section \ref{sec: setup}.) We assume $Z$ is stationary and ergodic with invariant density $p$. Equation \eqref{eq: D_defn} includes in particular the case when $D$ is smooth; in other words $D = \exp \pare{-\int_0^\cdot a (Z_t) \ud t}$, where $a$ represents a
short-rate function. However, the more general form of \eqref{eq: D_defn} is considered to accommodate a broader range of situations. For example:
\begin{itemize}
\item when payment streams are sometimes denominated in units of different account (for example, another currency, or financial assets), in which case discounting has to take into account the ``exchange rate''.
\item when, for pricing purposes, the payment stream, though denominated in domestic currency, must incorporate both traditional discounting and the density of the pricing kernel.
\end{itemize}

The two main results of the paper---Theorem \ref{thm: pde_soln} and Theorem \ref{thm:main}---identify the distribution of $(Z_0, X_0)$ in different ways. First, in the case where $\eta$ in \eqref{eq: D_defn} is non-degenerate and $f$ in \eqref{eq: X_def} is sufficiently regular, the conditional cumulative distribution function of $X_0$ given $Z_0$ is shown to coincide with the explosion probability of an associated locally elliptic diffusion and, hence, through the Feynman-Kac formula satisfies a partial differential equation (PDE): see Theorem \ref{thm: pde_soln}. Second, for general $\eta$ and $f$, using methods of diffusion time-reversal, we identify an ``ergodic'' process $(\ze,\chi)$ whose invariant distribution coincides with the joint distribution of $(Z_0,X_0)$. In particular, for any fixed starting point $x>0$ of $\chi$, the (random) empirical time-average law of $(\ze,\chi)$ on $[0,T]$ almost surely converges to the joint distribution of $(Z_0,X_0)$ in the weak topology: see Theorem \ref{thm:main}. The time-reversal result has the advantage of leading to an efficient method for obtaining the distribution via simulation, as the ergodic theorem enables estimation of the entire distribution based upon a single realization of $(\ze,\chi)$; a numerical example in Section \ref{S: num_ex} dramatically reinforces this point. However, it must be noted that the invariant distribution $p$ for $Z$ appears in the reversed dynamics, and hence must be known to perform simulation. When $Z$ is one-dimensional, or more generally, reversing, $p$ is given in explicit form with respect to the model parameters. In the general multi-dimensional setup, lack of knowledge of $p$ could pose an issue; however, we provide a potential way to amend the situation in the discussion after Theorem \ref{thm:main}. Note also that in the PDE result in Theorem \ref{thm: pde_soln}, explicit knowledge of $p$ is not necessary.

%The main results crucially depend upon two factors : the presence (or absence) of the independent Brownian motion $B$ (as expressed by the positivity of $\eta^2$ in \eqref{eq: D_defn}) and an explicit knowledge of the invariant density $p$ for $Z$.  When $\eta^2 > 0$ the conditional cdf $g$ coincides with the explosion time of a certain locally elliptic diffusion $(Z,\hat{X})$: see Theorem \ref{thm: pde_soln}. Thus, the classical Feynman-Kac formula implies that $g$ itself solves a certain PDE. Note that this yields a density for the law of $(Z_0,X)$.  Furthermore, when $\eta^2 > 0$ the reversed process $(\chi,\ze)$ is also locally elliptic: see Propositions \ref{prop: zeta_chi_t_dynamics}, \ref{prop: backwards_ergodicity}. Thus, the standard theory (see \cite{MR1326606}) applies for verifying ergodicity.

%When $\eta = 0$ it may be that the law of $X$ contains an atom : indeed, take $f = a \equiv 1$ and $\eta = \theta \equiv 0$ in \eqref{eq: X_def} and \eqref{eq: D_defn}.  Here, $X = 1$ with probability one.  Thus, one can not expect to identify a PDE for $g$ (at least in the classical, non-viscosity sense). However, the method of time reversal still allows for identification of the (now-degenerate) reversed process $(\ze,\chi)$.  Then, using classical results on ergodicity of stochastic dynamical systems (see \cite{MR748853,MR547492}), one may both obtain an explicit formula for the support of the law of $(Z_0,X)$ and prove that $(\ze,\chi)$ is ergodic.

\subsection*{Existing literature and connections}
Obtaining the distribution of the perpetuity $X_0$ is of great importance in the areas of finance and actuarial science; for this reason, perpetuities with a form similar to $X_0$ have been extensively studied. For example, \cite{MR1129194} deals with the case where
\begin{equation*}
X_0 = \int_0^\infty e^{-\sigma B_t - \nu t} \ud t,
\end{equation*}
establishing that $X_0$ has an inverse gamma distribution. This fits into the set-up of \eqref{eq: Z_defn}, \eqref{eq: D_defn} by taking $a=\nu-\sigma^2/2$, $f=1$, $\theta=0$ and $\eta = \sigma$.  Note that here $Z$ plays no role. In a similar manner, \cite{yor2001bessel,MR1211975,Delbaen_1993} consider the case
\begin{equation*}
X_0 = \int_0^\infty e^{-\int_0^t Z_u \ud u} \ud t; \qquad \ud Z_t = \kappa(\theta-Z_t) \ud t + \xi \sqrt{Z_t} \ud W_t; \qquad E =(0,\infty),
\end{equation*}
and obtain the first moment, along with bounds for other moments, of $X_0$. In \cite{MR1480643}, the perpetuity takes the form
\begin{equation}\label{eq: levy_perp}
X_0 = \int_0^\infty e^{-Q_t} d P_t, \quad 
\text{with } P \text{ and } Q \text{ being independent L\'{e}vy processes}.
\end{equation}
Under certain conditions on $P$ and $Q$, the distribution of $X_0$ is implicitly calculated by identifying the characteristic function and/or Laplace transform for $X_0$. In fact, the results of \cite{MR1480643} are pre-dated (for highly particular $P$ and $Q$), in \cite{paulsen1993327, MR1386179}. The Laplace transform method is also used in \cite{paulsen1999,MR1440822} to treat \eqref{eq: levy_perp} when $P_t = t$ and $Q$ is a diffusion. In addition to identifying a degenerate elliptic partial differential equation for the Laplace transform, they propose a candidate recurrent Markov chain whose invariant distribution has the law of $X_0$.  Lastly, the setup of \cite{MR1480643} is significantly extended in \cite{MR1833691} where, under minimal assumptions on $P$ and $Q$, the distribution of $X_0$ is shown to coincide with the unique invariant measure for a certain generalized Ornstein-Uhlenbeck process, a relationship that is confirmed in our current setting in Proposition \ref{prop: backwards_ergodicity}.

The use of time-reversal to identify the distribution of a discrete-time perpetuity is well known, dating at least back to \cite{Dufresne_1992}, where $X_0$ takes the form
\begin{equation*}
X_0 = \sum_{n=1}^\infty \left(\prod_{i=1}^n D_i\right) f_n,
\end{equation*}
where the discount factors $\pare{D_n}_{\nin}$ and cash flows $\pare{f_n}_{\nin}$ are two independent sequences of independent, identically distributed (iid) random variables.  To provide insight, the time-reversal argument in \cite{Dufresne_1992} is briefly presented here. With $X_0^{(N)} \dfn \sum_{n=1}^N \left(\prod_{i=1}^n D_i\right) f_n$ it is clear by the iid property that $X_0^{(N)}$ has the same distribution as $\widetilde{X}_N \dfn D_N f_N + D_N D_{N-1} f_{N-1} + .... + \left(\prod_{j=1}^N D_j\right)f_1$. Straightforward calculations show that  the reversed process $(\widetilde{X}_n)_{\nin}$ satisfies the recursive equation $\widetilde{X}_{n} = D_{n} \big(\widetilde{X}_{n-1} + f_{n} \big)$. Thus, assuming that $(\widetilde{X}_n)_{\nin}$ converges to a random variable $\widetilde{X}$ in distribution, $\widetilde{X}$ must solve the distributional equation $\widetilde{X} = D (\widetilde{X}+f)$, where $D$,  $f$ and $\widetilde{X}$ are independent, $D$ has the same law as $D_1$ and $f$ has the same law as $f_1$. In \cite{MR544194} solutions to the aforementioned distributional equation are obtained based upon the expectation of $\log(|D|)$ and $\log_+(|Df|)$. The tails of $\widetilde{X}$, as well as convergence of iterative schemes, are studied in \cite{MR1311930}; furthermore, \cite{MR1797309} gives ``almost'' if and only if conditions for the convergence of iterative schemes.

In a continuous time setting, we employ an argument similar in spirit, but rather different in execution, to \cite{Dufresne_1992}. Specifically, we extend $X_0$ to a whole ``forward'' process $X \dfn (1/D) \int_\cdot^\infty D_t f(Z_t) d t$ and then, for each $T>0$ define the reversed process $(\ze^T,\chi^T)$ on $[0,T$] by $\ze^T_t \dfn Z_{T-t}$, $\chi^T_t \dfn X_{T-t}$: see \eqref{eq: X_P_def}, \eqref{eq: ze_chi_def}.  Using results on time reversal of diffusions from \cite{MR866342} (alternatively, see \cite{MR814118,MR656280,MR687295,MR787590}), as well as additional elementary calculations, we obtain the dynamics for $(\ze^T,\chi^T)$. In fact, Proposition \ref{prop: zeta_chi_t_dynamics} shows the generator of $(\ze^T,\chi^T)$ does not depend upon $T$ and ergodicity can be studied for the process $(\ze,\chi)$ with the given generator. When $|\eta| >0$ in $E$  and $f$ is sufficiently regular, this generator is locally elliptic and the associated process $(\ze,\chi)$ is ergodic with invariant distribution equalling that of  $(Z_0,X_0)$: see Proposition \ref{prop: backwards_ergodicity}. In the general case a slightly weaker (but still sufficient) form of ergodicity still holds: starting $\ze$ off its invariant distribution $p$ and $\chi$ off any starting point $x>0$, the (random) empirical time-average laws of $(\ze,\chi)$ converge almost surely in the weak topology to the distribution of $(Z_0,X_0)$.

%To obtain the distribution of $(Z_0,X)$ the ergodic theorem on the reversed process can thus be used. However, as is well known (see \cite{MR866342}), the invariant density $p$ for $Z$ appears in the drift for the reversed process $\ze$ and hence simulation of $\ze$ requires knowledge of the gradient of the (log of) $p$.  Whereas this is not a problem in the single-dimensional or reversing case for $Z$, in the general multi-dimensional setting, $p$ is typically not explicitly known and this increases the complexity of the simulation procedure to estimate the invariant distribution of $(\ze,\chi)$.

\subsection*{Structure}
This paper is organized as follows: in Section \ref{sec: setup} we precisely state the given assumptions on the processes $Z$ and $D$, as well as the function $f$, paying particular attention to deriving sharp conditions under which $X_0$ is almost surely finite or infinite. The main results are then presented in Section \ref{sec: mr}. First, when $|\eta| > 0$ in $E$ and $f$ is sufficiently regular, the conditional cumulative distribution function of $X$ given $Z_0 =z$ is shown to satisfy a certain partial differential equation. Then, using the method of time reversal, we construct a probability space and diffusion $(\ze,\chi)$ such that with probability one its empirical time-average laws weakly converge to the joint distribution of $(Z_0,X_0)$ for all starting points of $\chi$.  Section \ref{sec: mr} concludes with a brief discussion how the distribution may be estimated via simulation, in particular proposing a method for obtaining the desired distribution when the invariant density $p$ for $Z$ is not explicitly known. Section \ref{S: num_ex} provides a numerical example in a specific case where the joint distribution of $(Z_0,X_0)$ is explicitly identifiable. Here, we compare the performance of the reversal method versus the direct method for obtaining the distribution of $X_0$. In particular we show that for a given desired level of accuracy (see Section \ref{S: num_ex} for a more precise definition), the method of time reversal is approximately $175$ to $300$ times faster than the direct method. The remaining sections contain the proofs: Section \ref{sec: X0_finite} proves the statements regarding the finiteness of $X_0$; Section \ref{sec: proof_of_pde} proves the partial differential equation result; Section \ref{sec: tr} obtains the dynamics for the time-reversed process $(\ze,\chi)$; Section \ref{sec: ze_chi_ergodic} proves the (weak) ergodicity with the correct invariant distribution. Finally, a number of technical supporting results are included in the appendix.

\section{Problem Setup}\label{sec: setup}

\subsection{Well-posedness and ergodicity}
The first order of business is to specify precise coefficient assumptions so that $Z$ in \eqref{eq: Z_defn} and $D$ in \eqref{eq: D_defn}, are well-defined. As for $Z$, we work in the standard locally elliptic set-up for diffusions: for more information, see \cite{MR1326606}. Let $E\subseteq\Real^d$ be an open, connected region. We assume the existence of $\gamma \in (0, 1]$ such that:
\begin{enumerate}
	\item[(A1)] there exists a sequence of regions $(E_n)_{\nin}$ such that $E = \bigcup_{n=1}^{\infty} E_n$, each $E_n$ being open, connected, bounded, with $\partial E_n$ being $C^{2,\gamma}$ and satisfying $\bar{E}_n\subset E_{n+1}$ for all $\nin$.
	\item[(A2)] $m\in C^{1,\gamma}(E;\Real^d)$ and $c\in C^{2,\gamma}(E;\spos^d)$, where $\spos^d$ is the space of symmetric and strictly positive definite $(d \times d)$-dimensional matrices.
\end{enumerate}

With the provisos in (A1) and (A2), define $L^Z$ as the generator associated to $(m,c)$:
\begin{equation}\label{eq: LZ}
L^Z \dfn \frac{1}{2}\sum_{i,j=1}^d c^{ij}\partial^2_{ij} + \sum_{i=1}^d m^i\partial_i.\ \footnote{In the sequel the summands will be omitted using Einstein's convention; therefore, $L^Z$ is written as $L^Z = (1/2)c^{ij}\partial^2_{ij} + m^i\partial_i$.}
\end{equation}
Under (A1) and (A2), one can infer the existence of a solution to the martingale problem for $L^Z$ on $E$, with the possibility of explosion to the boundary of $E$ : see \cite{MR1326606} We wish for something stronger; namely, to construct a filtered probability space $\basisp$ on which there is a strong, stationary, ergodic solution to the SDE in \eqref{eq: Z_defn} with invariant density $p$. In \eqref{eq: Z_defn}, $W$ is a $d$-dimensional Brownian Motion and $\sigma = \sqrt{c}$, the unique positive definite symmetric matrix such that $\sigma^2 = c$.  In order to achieve this, we ask that

\begin{enumerate}
	\item[(A3)] The martingale problem for $L^Z$ on $E$ is well posed and the corresponding solution is recurrent.  Furthermore, there exists a strictly positive $p\in C^{2,\gamma}(E,\Real)$ with $\int_E p(z) \ud z = 1$ satisfying $\tilde{L}^Z p = 0$, where $\tilde{L}^Z$ is the formal adjoint of $L^Z$ given by
\begin{equation}\label{eq: tildeLZ}
\tilde{L}^Z \dfn \frac{1}{2}c^{ij}\partial^2_{ij} -\left(m^i-\partial_jc^{ij}\right)\partial_i -\left(\partial_i m^i -\frac{1}{2}\partial^2_{ij} c^{ij}\right).
\end{equation}
\end{enumerate}

%\begin{ass}\label{ass: ergodic}
%
%The martingale problem for $L^Z$ on $E$ is well posed and the corresponding solution
%  $(\prob^{Z,z})_{z\in E}$ is recurrent.  Furthermore, there exists a strictly
%  positive $p\in C^{2,\gamma}(E,\Real) \cap L^1(E,\Leb)$ solving $\tilde{L}^Z p = 0$.  Here,
%  $\tilde{L}^Z$ is the formal adjoint of $L^Z$, and takes the form:
%\begin{equation}\label{eq: tildeLZ}
%\tilde{L}^Z \dfn \frac{1}{2}c^{ij}\partial^2_{ij} -\left(m^i-c^{ij}_j\right)\partial_i -\left(m^i_i -
%\frac{1}{2}c^{ij}_{ij}\right).
%\end{equation}
%\end{ass}

We summarize the situation in the following result: the extra Brownian motion $B$ in its statement will be used to define the process $D$ via \eqref{eq: D_defn} later on.

\begin{thm} \label{thm: Z_existence}
Under assumptions (A1), (A2) and (A3), there exists a filtered probability space $\basisp$ satisfying the usual conditions supporting two independent Brownian motions $W$ and $B$, $d$-dimensional and $k$-dimensional respectively, such that $Z$ satisfies \eqref{eq: Z_defn} and is stationary and ergodic with invariant density $p$.
\end{thm}

\begin{rem}
According to \cite[Corollary 5.1.11]{MR1326606}, in the one-dimensional case, where $E=(\alpha,\beta)$ for $-\infty\leq\alpha < \beta\leq\infty$, the above assumption (A3) is true if and only if  for some $z_0\in E$
\begin{align*}
\int_\alpha^{z_0} \exp \pare{-2\int_{z_0}^{z}\frac{m(s) }{c(s)} \ud s} \ud z &= \infty, \\
\int_{z_0}^{\beta} \exp \pare{-2\int_{z_0}^{z}\frac{m(s)}{c(s)} \ud s}\ud z  &= \infty, \\
\int_{\alpha}^{\beta}\frac{1}{c(z)} \exp \pare{2\int_{z_0}^{z}\frac{m(s)}{c(s)} \ud s}\ud z &< \infty.
\end{align*}
In this case, it holds that
\[
p(z) = Kc^{-1}(z)\exp\pare{2\int_{z_0}^z \frac{m(s)}{c(s)} \ud s}, \quad z \in (\alpha, \beta),
\]
where $K > 0$ is a normalizing constant.

In the multi-dimensional case, suppose that there exists a function $H: E \mapsto \Real$ with the property that $c^{-1}(2m - \dvg{c}) = \nabla H$, where $\dvg{c}$ is the (matrix) divergence defined by\footnote{This definition is equivalent to the standard definition of divergence for matrices, where the divergence operator is applied to the columns, by the symmetry of $c$. Also, to differentiate the matrix divergence from its vector counterpart, we will write $\dvg{A}$ for symmetric matrices $A$ and $\nabla\cdot v$ for vector valued functions $v$.} $\dvg{c}^i = \partial_jc^{ij}, i = 1,...,d$. Then, $Z$ is a reversing Markov process. Furthermore Assumption (A3) follows if it can be shown that $Z$ does not explode to the boundary of $E$ and $K \dfn \int_E \exp(H(z)) \ud z < \infty$. Indeed, by construction $p =  e^H/K$ satisfies $\tilde{L}^Zp = 0$, $\int_E p(z)\ud z = 1$. Thus, if $Z$ does not explode, it follows from \cite[Theorem 2.8.1, Corollary 4.9.4]{MR1326606} that $Z$ is recurrent. In fact, $Z$ is ergodic, as shown in \cite[Theorems 4.3.3, 4.9.5]{MR1326606}. Absent the reversing case, there are many known techniques for checking ergodicity---see \cite{MR0494525, MR1326606}. For example, if there exist a smooth function $u: E \mapsto \Real$, an integer $N$ and constants $\eps > 0$ and $C > 0$ such that $L^Zu\leq -\eps$ and $u\geq -C$ on $E\setminus E_N$, then (A3) holds.
\end{rem}

In order to ensure that $D$ in \eqref{eq: D_defn} is well defined, we assume that
\begin{enumerate}
	\item[(A4)] $a\in C^{1,\gamma}(E;\Real_+)$, $\eta\in C^{2,\gamma}(E; \Real^k)$, and $\theta\in C^{2,\gamma}(E; \Real^d)$.
\end{enumerate}
Given (A4) and all previous assumptions, it follows that \eqref{eq: D_defn} possesses a strong solution on $\basisp$ of Theorem \ref{thm: Z_existence}; in fact, defining $R \dfn - \log(D)$, it holds that
\begin{equation}\label{eq: R_defn}
R = \int_0^\cdot \pare{a  + \frac{1}{2}\pare{\theta'c\theta + |\eta|^2 }} (Z_t)
\ud t +  \int_0^\cdot \theta(Z_t)' \sigma(Z_t)\ud W_t +  \int_0^\cdot \eta(Z_t)'\ud B_t.
\end{equation}

%Under the given assumptions there exists a filtered probability space $\basisp$ satisfying the usual conditions supporting independent $d$ and $k$ dimensional Brownian Motions $W,B$ such that $(Z,D)$ are strong
%solutions of \eqref{eq: Z_defn} and \eqref{eq: D_defn} respectively. Furthermore, $Z$ is stationary ergodic under $\prob$ with invariant density $p$. If $\prob^z$ is defined by
%\begin{equation}\label{eq: pz_def}
%\prob^z\bra{A}:=\prob\bra{A\such  Z_0 = z},\qquad A\in\F,
%\end{equation}
%then for all $A\in \F$, $\prob\bra{A} = \int_E p(z)\prob^z\bra{A}dz$.

\subsection{Finiteness of $X_0$}\label{subsec: X_finite}

Having the set-up for the existence of $Z$ and $D$, we proceed to $X_0$. For the time being, we shall just assume that the function $f : E \mapsto \Real_+$ is in $\Lb^1(E,p)$\footnote{We define $\Lb^1(E,p)$ to be those Borel measurable functions $g$ on $E$ so that $\int_D |g(z)|p(z)dz < \infty$. Thus, Borel measurability is implicitly assumed throughout.}. For the PDE results of Theorem \ref{thm: pde_soln} below we require a slightly stronger regularity assumption on $f$, though the time-reversal results of Theorem \ref{thm:main} make no additional assumptions. Now, for $f$ not necessarily in $\Lb^1(E,p)$, it is entirely possible that $X_0$ takes infinite values with positive probability. In this section, conditions are given under which $\prob\bra{X_0 < \infty} = 1$ or, conversely, when $\prob\bra{X_0 < \infty} = 0$.

%The proofs of the statements below are given in the Appendix.

\begin{lem}\label{lem: X_finite_A}
Let (A1), (A2), (A3) and (A4) hold. For the invariant density $p$ of $Z$, assume there exists $\eps > 0$ such that
\begin{equation}\label{eq: eps_finite_cond}
 \left(a+\frac{1-\eps}{2}(\theta'c\theta + \eta'\eta)\right)_{-} \in \Lb^1(E, p), \text{ and } \int_E\left(a+\frac{1-\eps}{2}(\theta'c\theta
  +\eta'\eta)\right)(z) \, p(z) \ud z  > 0.
\end{equation}
Then, the following hold:
\begin{enumerate}[i)]
\item There exists $\kappa > 0$ such that for all $z\in E$, $\prob\bra{\lim_{t\to\infty} e^{\kappa t} D_t = 0\such Z_0 =z}=1$. In particular, $\lim_{t\to\infty} e^{\kappa t}D_t = 0$ $\prob$ a.s..
\item For any $f\in \Lb^1 (E, p)$, it holds that $\prob\bra{X_0 < \infty}= 1$.
\end{enumerate}

\end{lem}

\begin{rem} Note that \eqref{eq: eps_finite_cond} holds if $a>0$ on $E$. The more complicated form in \eqref{eq: eps_finite_cond} allows $a$ to take (unbounded) negative values. Furthermore, in the case where $\left(\theta'c\theta + \eta'\eta\right) \in \Lb^1(E, p)$ then equation \eqref{eq: eps_finite_cond} is equivalent to:
\begin{equation}\label{eq: no_eps_finite_cond}
\left(a+\frac{1}{2}(\theta'c\theta + \eta'\eta)\right)_{-}\in \Lb^1(E,p), \text{ and } \int_E\left(a+\frac{1}{2}(\theta'c\theta
  +\eta'\eta)\right)(z) \, p(z) \ud z  > 0.
\end{equation}
\end{rem}

As a partial converse to Lemma \ref{lem: X_finite_A} we have

\begin{lem}\label{lem: X_infinite_A}
Let (A1), (A2), (A3) and (A4) hold. For the invariant density $p$ of $Z$, assume there exists $\eps > 0$ such that
\begin{equation}\label{eq: eps_infinite_cond}
 \left(a+\frac{1 + \eps}{2}(\theta'c\theta + \eta'\eta)\right)_{+} \in \Lb^1(E, p), \text{ and } \int_E\left(a+\frac{1 + \eps}{2}(\theta'c\theta
  +\eta'\eta)\right)(z) \, p(z) \ud z  \leq 0.
\end{equation}
(If $\theta'c\theta + \eta'\eta\equiv 0$, then assume that $a_+\in \Lb^1(E,p)$ and $\int_E a(z) p(z) \ud z < 0$.) If $f$ is such that $\int_E f(z) p(z) \ud z > 0$, then $\prob\bra{X_0 < \infty} = 0$.
\end{lem}

\begin{rem}
Let (A1), (A2), (A3) and (A4) hold, and assume that $a$ is non-negative. A combination of Lemma \ref{lem: X_finite_A} and Lemma \ref{lem: X_infinite_A} yield sharp conditions for the finiteness of $X_0$ that do not require knowledge of $p$, at least for bounded $f$.
\begin{itemize}
	\item If $a + (1/2)(\theta'c\theta + \eta'\eta) \not\equiv 0$, then
  $\prob\bra{X_0 < \infty} = 1$ holds if $f\in \Lb^1(E,p)$.
	\item If $ a + (1/2)(\theta'c\theta + \eta'\eta) \equiv 0 $ then $\prob\bra{X_0 < \infty} = 0$ holds if $\int_E f(z) p(z) \ud z > 0$.
\end{itemize}
\end{rem}

In view of Lemma \ref{lem: X_finite_A}, we ask that

\begin{enumerate}
	\item[(A5)] $f \in \Lb_+^1(E,p)$, $\int_E f(z) p(z) \ud z > 0$, and there exists $\eps > 0$ such that
\begin{equation}\label{eq: eps_finite_cond_2}
 \left(a+\frac{1-\eps}{2}(\theta'c\theta + \eta'\eta)\right)_{-} \in \Lb^1(E, p), \text{ and } \int_E\left(a+\frac{1-\eps}{2}(\theta'c\theta
  +\eta'\eta)\right)(z) \, p(z) \ud z  > 0.
\end{equation}
\end{enumerate}

%\begin{ass}\label{ass: X_finite}
%$f\in L^1(E,p), f\not\equiv 0$ and the condition in \eqref{eq: eps_finite_cond} holds: i.e. for some $\eps > 0$:
%\begin{equation}\label{eq: eps_finite_cond_2}
%\begin{split}
%&\left(a+\frac{1-\eps}{2}(\theta'c\theta + \eta'\eta)\right)^{-}\in
%L^1(E,p);\qquad \int_E\left(a+\frac{1-\eps}{2}(\theta'c\theta
%  +\eta'\eta)\right)dp  > 0.
%\end{split}
%\end{equation}
%Thus, by Lemma \ref{lem: X_finite_A}, $\prob\bra{X<\infty} = 1$ and there exists a $\kappa > 0$ so that for all $z\in E$, $\lim_{t\to\infty} e^{\kappa t}D_t = 0$ $\prob^z$ almost surely.
%\end{ass}

To recapitulate, for the remainder of the article the following is assumed:

\begin{ass}\label{ass: standing}
We enforce throughout all above assumptions (A1), (A2), (A3), (A4) and (A5).
\end{ass}

\section{Main Results}\label{sec: mr}

\subsection{The distribution of $X_0$ via a partial differential equation}\label{subsec: pde}

Define the cumulative distribution function $g$ of  $X_0$ given $Z_0$ by
\begin{equation}\label{eq: cond_cdf}
g(z,x) \dfn  \prob\bra{X_0 \leq x \such  Z_0 = z}, \quad (z, x) \in F\dfn E \times (0, \infty).
\end{equation}
Next, recall that Assumption \ref{ass: standing} implies that $Z_0$ has a density $p$, and define the joint distribution $\pi$ of $(Z_0,X_0)$ by
\begin{equation}\label{eq: pi_def}
\pi(A) \dfn  \iint_A p(z)g(z,\ud x)\ud z;\qquad A\in\mathcal{B}(F).
\end{equation}

Under Assumption \ref{ass: standing}, as well as an additional smoothness requirement on $f$ and non-degeneracy requirement on $\eta$, the first main result (Theorem \ref{thm: pde_soln} below) shows $g$ solves a certain PDE on the state space $F$. This will imply that the joint distribution of $(Z_0, X_0)$ has a density (still labeled $\pi$) and the law of $X_0$ charges all of $(0,\infty)$.

To motivate the result, as well as to fix notation, for each $x \in (0, \infty)$, consider the process
\begin{equation}\label{eq: hat_x_process}
Y^x \dfn \frac{1}{D }\left(x -
  \int_0^\cdot D_t f(Z_t)\ud t\right).
\end{equation}
Since Assumption \ref{ass: standing} implies $\prob\bra{\lim_{t\to\infty}D_t = 0\such Z_0 =z} =1$ for all $z\in E$, it is clear that given $Z_0 = z$, on $\cbra{X_0 < x}$ the process $Y^x$ tends to $\infty$.  Alternatively, on $\cbra{X_0 > x}$, $Y^x$ will hit $0$ at some finite time. What happens on $\cbra{X_0 = x}$ is not immediately clear but it will be shown under the given assumptions there is no probability of this occurring. For fixed $(z,x) \in F$, it follows that $1-g(z,x)$ equals the probability that $Y^x$ hits zero, given $Z_0 = z$.  According the Feynman-Kac formula, such probabilities ``should'' solve a PDE.  To identify the PDE, note that the joint equations governing $Z$ and $Y^x$ are
\begin{align*}
Z &= Z_0 + \int_0^\cdot m(Z_t)\ud t + \int_0^\cdot \sigma(Z_t)\ud W_t, \\
Y^x &= x + \int_0^\cdot \left(-f(Z_t) + Y^x_t\left(a(Z_t) + \theta'c\theta(Z_t)
    + \eta'\eta(Z_t)\right)\right)\ud t \\
& \quad + \int_0^\cdot Y^x_t\left(\theta'\sigma(Z_t)\ud W_u +
  \eta(Z_t)'\ud B_t\right).
\end{align*}
Define $b:F\mapsto\Real^{d+1}$ and $A: F\mapsto
\spos^{d+1}$ by
\begin{equation}\label{eq: bA_def}
b(z,x) \dfn \left(\begin{array}{c} m(z) \\ -f(z) +
    x\left(a+\theta'c\theta+\eta'\eta\right)(z)\end{array}\right);\quad A(z,x)
\dfn \left(\begin{array}{c c} c(z) & x c\theta(z) \\ x\theta'c(z)
    & x^2\left(\theta'c\theta + \eta'\eta\right)(z)\end{array}\right),
\end{equation}
for all $(z,x) \in F$. Note that if, in addition to Assumptions \ref{ass: standing}, $|\eta|(z) > 0, z\in E$ then $A$ is locally elliptic. Let $L$ be the second order differential operator associated to $(A,b)$, i.e.,
\begin{equation}\label{eq: L_def}
L\dfn \frac{1}{2}A^{ij}\partial^2_{ij} + b^i \partial_i.
\end{equation}
Note that $L\phi = L^{Z}\phi$ for functions $\phi$ of $z\in E$ alone.  With the previous notation, the first main result now follows.

\begin{thm}\label{thm: pde_soln}
Let Assumptions \ref{ass: standing} hold, and suppose further that a) $f\in C^{1,\gamma}(E;\Real_+)$ and b) $|\eta (z)| > 0$ for all $z\in E$. Then, $g \in C^{2,\gamma}(F)$ satisfies $Lg = 0$ with the following ``locally uniform'' boundary
conditions
\begin{equation}\label{eq: g_bc}
\lim_{n\to\infty} \sup_{x\leq n^{-1} ,z\in E_{k}} g(z,x) = 0;\qquad
\lim_{n\to\infty} \inf_{x\geq n ,z\in E_{k}} g(z,x) = 1, \quad \forall \kin.
\end{equation}
Furthermore, $g$ is unique within the class of solutions to $Lg = 0$ taking values in $[0,1]$  with the above boundary conditions.
\end{thm}

\begin{rem}\label{rem: non-degeneracy}
The non-degeneracy assumption on $\eta$ is essential for the existence of a density; if $\eta \equiv 0$ it may be that the distribution of $X_0$ has an atom. Indeed, take $f \equiv 1$, $a \equiv 1$, $\eta \equiv 0$, $\theta \equiv 0$. Then, $X_0 = \int_0^\infty e^{-t}\ud t = 1$ with probability one.
\end{rem}

\begin{rem}
Theorem \ref{thm: pde_soln} implies the law of $X_0$ charges all of $(0,\infty)$, even for those functions $f$ which are bounded from above. Theorem \ref{thm: pde_soln} also implies that $X_0$ has a density without imposing Hormander's condition \cite[Chapter 2]{MR2200233} on the coefficients in \eqref{eq: bA_def}.  Rather, the infinite horizon combined with the presence of the independent Brownian motion $B$ ``smooth out'' the distribution of $X_0$.

\end{rem}

Theorem \ref{thm: pde_soln} is certainly important from a theoretical viewpoint. However, it appears to be of limited practical use. Even under the force of the extra non-degeneracy condition $|\eta| > 0$, it is unclear how to numerically solve the PDE $Lg = 0$ with the given boundary conditions \eqref{eq: g_bc}, as there are no natural auxiliary boundary conditions in the spatial domain of $z \in E$. In Subsection \ref{subsec: time_reverse} that follows we provide an alternative, more useful method for estimating numerically the law of $(Z_0, X_0)$.

\subsection{The distribution of $(Z_0, X_0)$ via diffusion time-reversal}\label{subsec: time_reverse}

The goal here is to show that the distribution of $(Z_0,X_0)$ coincides with the
invariant distribution of a positive recurrent process $(\ze,\chi)$. In order to see the connection, extend $X_0$ to a whole process $(X_t)_{\tir}$ defined via
\begin{equation}\label{eq: X_P_def}
X \dfn \frac{1}{D} \int_\cdot^\infty D_t f(Z_t) \ud t,
\end{equation}
and note that $(Z_t, X_t)_{\tir}$ is a stationary process under $\prob$. Fix $T>0$, and define the process $(\ze^T_t, \chi^T_t)_{t \in [0, T]}$ via
time-reversal:
\begin{equation}\label{eq: ze_chi_def}
\ze^T_t \dfn Z_{T - t};\qquad \chi^T_t \dfn X_{T -t};\qquad t \in
[0, T].
\end{equation}
It still follows that $(\ze^T, \chi^T)$ is stationary under $\prob$, with the same one-dimensional marginal distribution as $(Z_0, X_0)$. Furthermore, stationarity of $(Z, X)$ clearly implies that the law of the process $(\ze^T, \chi^T)$ does not depend on $T$ (except for its time-domain of definition). Therefore, one may create a new process $(\zeta_t, \chi_t)_{\tir}$ such that the law of $(\zeta^T, \chi^T)$ is the same as the law of $(\zeta_t, \chi_t)_{t \in [0, T]}$ for all $t \in T$. If one can establish that $(\ze,\chi)$ is ergodic, then the distribution of $(Z_0,X_0)$ may be efficiently estimated via the ergodic theorem.

Towards this end, one needs to understand the behavior of $(\ze,\chi)$. Standard results (e.g. \cite{MR866342}) in the theory of time-reversal imply that $\ze$ is a diffusion in its own filtration, and identify the corresponding coefficients. In order to deal with $\chi$, we return to the definition of $\chi^T$ and define yet one more process $(\Delta^T_t)_{t \in [0, T]}$ via
\begin{equation}\label{eq: Delta_def}
\Delta^T_t = \frac{D_T}{D_{T-t}}, \quad t \in [0, T].
\end{equation}
Using all previous definitions, we obtain that
\begin{align}
\nonumber \chi^T_t = X_{T - t} &= \frac{1}{D_{T - t}} \int_{T - t}^\infty D_u f(Z_u) \ud u \\
\nonumber &= \frac{D_T}{D_{T - t}} \pare{X_T + \int_{T - t}^T \frac{D_u}{D_T} f(Z_u) \ud
  u} \\
&= \Delta_t^T \pare{\chi^T_0 + \int_0^t \frac{1}{\Delta^T_u} f(\ze^T_u) \ud u}, \quad t \in [0, T]. \label{eq: chi_from_delta}
\end{align}
As it turns out, one can describe the joint dynamics of $(\zeta^T, \Delta^T)$ in appropriate filtrations (and these dynamics do not depend on $T$, as expected). To ease the presentation, recall from Section \ref{sec: setup} that for any $\spos^d$ valued smooth function $A$ on $E$ the (matrix) divergence is defined by $\dvg{A}^i = \partial_j A^{ij}$ for $i=1,...,d$. It is then shown in Section \ref{sec: tr} that $(\zeta^T, \Delta^T)$ is such that
\begin{align*}
\ze^T & = \ze^T_0 + \int_0^\cdot\left(c\frac{\nabla p}{p} + \dvg{c} - m\right)(\ze^T_t)\ud t + \int_0^\cdot \sigma(\ze^T_t)\ud W^T_t, \\
\Delta^T &= 1 + \int_0^{\cdot}\Delta^T_t\pare{\theta' c \frac{\nabla p}{p} + \nabla\cdot(c\theta) - a}(\ze^T_t) \ud t + \int_0^\cdot
\Delta^T_t \pare{\eta(\ze^T_t)'\ud B^T_t + \theta'\sigma(\ze^T_t)\ud W^T_t} \\
&= 1 + \int_0^{\cdot}\Delta^T_t\pare{\theta' (m - \dvg{c}) + \nabla\cdot(c\theta) - a}(\ze^T_t) \ud t + \int_0^\cdot
\Delta^T_t \pare{\eta(\ze^T_t)'\ud B^T_t + \theta (\ze^T_t)' \ud \zeta^T_t}
\end{align*}
for independent Brownian motions $(W^T, B^T)$ in an appropriate filtration.

From the joint dynamics of $(\zeta^T, \Delta^T)$ one obtains the joint dynamics of $(\zeta^T, \chi^T)$, which again do not depend on $T$. In particular, since $\Delta^T$ is a
semimartingale, \eqref{eq: chi_from_delta} yields that
\begin{equation} \label{eq: chitT_dyn}
\begin{split}
\zeta^T &= \zeta^T_0 + \int_0^\cdot \left(c\frac{\nabla p}{p} + \dvg{c} - m\right)(\ze^T_t)\ud t + \int_0^\cdot \sigma(\ze^T_t)\ud W^T_t\\
\chi^T &= \chi^T_0 + \int_0^\cdot \left(f(\ze^T_t) - \chi^T_t\left(a-\theta'c\frac{\nabla p}{p}-\nabla\cdot(c\theta)\right)(\ze^T_t)\right)\ud t\\
&\qquad\qquad + \int_0^\cdot \chi^T_t\left(\eta(\ze^T_t)'\ud B^T_t + \theta'c(\ze^T_t)'\ud W^T_t\right).
\end{split}
\end{equation}

For a generic version $(\ze,\chi)$ with the same generator (which does not depend upon time) as $(\ze^T,\chi^T)$ above, ergodicity of $Z$ implies ergodicity of $\ze$ (see Proposition \ref{prop: zeta_t_dynamics} later on in the text).  Furthermore, $\chi$ is ``mean reverting'' as can easily be seen when $\theta \equiv 0$, and $a> 0$, and continues to be true in the general case.  Thus, one expects the empirical laws of $(\ze,\chi)$ to satisfy a certain strong law of large numbers, an intuition that is made precise in the following result.

\begin{thm} \label{thm:main}
Let Assumption \ref{ass: standing} hold. Then, there exists a probability space $(\Omega, \bF, \qprob)$ supporting independent $d$ and $k$ dimensional Brownian motions $W$ and $B$, as well as process $\zeta$ satisfying
\[
\ze = \ze_0 + \int_0^\cdot \left(c\frac{\nabla p}{p} + \dvg{c} - m\right)(\ze_t)\ud t + \int_0^\cdot \sigma(\ze_t)\ud W_t,
\]
where $\zeta_0$ is an $\F_0$-measurable random variable with density $p$.

Define the process $\Delta$ as the solution to the linear differential equation
\begin{equation}\label{eq: Delta_no_T_def}
\Delta = 1 + \int_0^\cdot \Delta_t\pare{\theta' (m - \dvg{c}) + \nabla\cdot(c\theta) - a}(\ze_t) \ud t + \int_0^\cdot
\Delta_t \pare{\eta(\ze_t)'\ud B_t + \theta (\ze_t)' \ud \zeta_t},
\end{equation}
and then, for any $x \in (0, \infty)$, define $\chi^x$ as the solution to the linear differential equation
\begin{equation}\label{eq: chi_no_T_def}
\chi^x = x + \int_0^\cdot \chi^x_t \, \frac{\ud \Delta_t} {\Delta_t} + \int_0^\cdot f(\ze_t) \ud t.
\end{equation}
Lastly, let $x \in (0, \infty)$,  $T \in (0, \infty)$ and set $F=E\times(0,\infty)$ as in \eqref{eq: cond_cdf}. Define the (random) empirical measure $\pih^x_T$ on $\B (F)$, the Borel subsets of $F$ by
\begin{equation}\label{eq: LT_def}
\pih^x_T \bra{A} \dfn \frac{1}{T} \int_0^T \indic_A (\zeta_t, \chi^x_t) \ud t, \quad A \in \B (F).
\end{equation}
With the above notation, there exists a set $\Omega_0 \in \F_\infty$ with $\qprob \bra{\Omega_0}= 1$ such that
\begin{equation}\label{eq: main_pi_conv}
\lim_{T \to \infty} \pih^x_T (\omega) = \pi \text{ weakly,  for all } x \in (0, \infty) \text{ and } \omega \in \Omega_0,
\end{equation}
where $\pi$ is the joint distribution of $(Z_0,X_0)$ under $\prob$ given in \eqref{eq: pi_def}.
\end{thm}

\begin{rem} \label{rem: chi_via_zeta}
In the context of Theorem \ref{thm:main}, note that the processes $\Delta$ and $\chi^x$ can be given in closed form in terms of $\zeta$; indeed,
\begin{align*}
\Delta &= \exp \pare{ \int_0^\cdot \pare{\theta' (m - \dvg{c}) + \nabla\cdot(c\theta) - a}(\ze_t) \ud t}\mathcal{E}\left(\int_0^\cdot \pare{\eta(\ze_t)'\ud B_t + \theta (\ze_t)' \ud \zeta_t}\right)_\cdot, \\
\chi^x &= \Delta \pare{x + \int_0^\cdot \frac{1}{\Delta_t} f(\ze_t) \ud t}, \quad x \in (0, \infty).
\end{align*}
\end{rem}

Theorem \ref{thm:main} provides a way to efficiently estimate the joint distribution of $(Z_0,X_0)$ efficiently through Monte-Carlo simulation. Indeed, one need only obtain a single path of the reversed process $(\ze,\chi^x)$ to recover the distribution $\pi$. However, the applicability of the result above depends heavily on whether or not the distribution $p$ for $Z_0$ is known, as it (together with its gradient) appears in the dynamics of $\zeta$. In the case where $Z$ is one-dimensional, or more generally, reversing, $p$ can be expressed in closed form from the model coefficients $m$ and $c$ in the dynamics for $Z$. Furthermore, there are certain cases of non-reversing, multi-dimensional diffusions, where $p$ can be (semi-)explicitly computed, as the next example shows.

\begin{exa} \label{exa: OU}
Assume that $Z$ is a multi-dimensional Ornstein-Uhlenbeck process with dynamics
\begin{equation*}
dZ_t = -\gamma(Z_t-\Theta)dt + \sigma dW_t, \quad \tir,
\end{equation*}
where $\gamma \in \Real^{d\times d}$, $\Theta\in \Real^d$, and $\sigma\in\Real^{d\times d}$. Here, $E = \Real^d$ and (A1) clearly holds. Furthermore (A2) is satisfied when $c = \sigma\sigma'$ is (strictly) positive definite; in fact, we take $\sigma$ as the unique positive definite square root of $c$. The process $Z$ need not be reversing, as can clearly be seen when $\sigma$ is the identity matrix, $\Theta = 0$ and $\gamma$ is not symmetric.  However, as will be argued below, the ergodic assumption (A3) holds when all eigenvalues of $\gamma$ have strictly positive real part, and one may identify the invariant density ``almost'' explicitly. To see this, a direct calculation shows that if a symmetric matrix $J$ satisfies the Riccati equation
\begin{equation}\label{eq: riccati}
JJ = \sigma \gamma'\sigma^{-1} J + J\sigma^{-1}\gamma\sigma,
\end{equation}
then the function
\begin{equation*}
p(z) = \exp\pare{-\frac{1}{2}(z-\Theta)'\sigma^{-1} J \sigma^{-1}(z-\Theta)}, \quad z \in \Real^d,
\end{equation*}
satisfies $\tilde{L}^Z p = 0$ where $\tilde{L}^Z$ is as in \eqref{eq: tildeLZ}.  If $J$ is additionally positive definite then, up to a normalizing constant, $p$ is the density for a normal random variable with mean $\Theta$ and covariance matrix $\Sigma = \sigma J^{-1}\sigma$. Thus, $p$ is integrable on $\Real^d$ and (A3) follows from \cite[Corollary 4.9.4]{MR1326606} which proves recurrence for $Z$.

It thus remains to construct a symmetric, positive definite solution to \eqref{eq: riccati}.  From \cite[Lemma 2.4.1, Theorem 2.4.25]{MR1997753} such a solution (called the ``stabilizing solution'' therein) exits if a) the pair $(\sigma^{-1}\gamma\sigma, 1_d)$ is \emph{stabilizable}, in that there exists a matrix $F$ such that $\sigma^{-1}\gamma\sigma - F$ has eigenvalues with strictly negative real part and b) the eigenvalues of $\sigma^{-1}\gamma\sigma$ have strictly positive real part.  In the present case, each of these statements readily follows: for the first statement, one can take $F = \sigma^{-1}\gamma\sigma + 1_d$; for the second statement, note that the eigenvalues of $\sigma^{-1}\gamma\sigma$ coincide with those of $\gamma$, which by assumption have strictly positive real part. Therefore, even in this non-reversing case one may still identify $p$.

\end{exa}

The previous interesting Example \ref{exa: OU} notwithstanding, for non-reversing, multi-dimensional diffusions, even after verifying the ergodicity of $Z$ (and hence the existence of $p$) one does not typically know $p$ explicitly. In such cases, the following simulation method is proposed: fix a large enough $T$ and first simulate $(Z_t)_{t \in [0, 2T]}$ via \eqref{eq: Z_defn}, starting from \emph{any} point $Z_0$ (since the invariant density is unknown). If the choice of $T$ is large enough, the process $(Z_t)_{t \in [T,2T]}$ will behave as the stationary version in \eqref{eq: Z_defn}, since $Z_T$ will have approximately density $p$. In that case, defining $(\ze_t)_{t\in [0, T]}$ via $\ze_t = Z_{2 T - t}$ for $t \in [0, T]$, $\ze$ should behave as it should in the dynamics \eqref{eq: ze_chi_dyn}, even with $\zeta_0$ having (approximate) density $p$. Now, given $\ze$, $\chi^x$ may be defined via the formulas of Remark \ref{rem: chi_via_zeta}; therefore, for large enough $T$, the empirical measure $\pih^x_T$ should approximate in the weak sense the joint law $\pi$.

Note finally that when $p$ is known and $|\eta| > 0$, and under certain mixing conditions (see \cite{MR1161349,Veretennikov_Bounds}), one can also obtain uniform estimates for the speed at which the above convergence takes place.

\begin{rem}\label{rem: Kliemann}

In the case when $\theta = \eta \equiv 0$ and $f\in C^{1,\gamma}(E;\Real_+)$, it is possible to explicitly identify the support of $\pi$.  Such an identification follows from more general ergodic results on ``stochastic differential systems'' obtained in \cite{MR748853, MR1723992}. To identify the support, note that when $\theta=\eta\equiv 0$, it follows that $\Delta_t = \exp\pare{-\int_0^T a(\ze_u)du}$.  A direct calculation using Remark \ref{rem: chi_via_zeta} shows that $\chi^x$ has dynamics
\begin{equation}\label{eq: ze_chi_dyn_deg}
\begin{split}
\ud \chi^x_t& = \pare{f(\ze_t)-\chi^x_t a(\ze_t)} \ud t.\\
\end{split}
\end{equation}
Hence, the paths of $\chi^x$ are of bounded variation. Now, define
\begin{equation}\label{eq: l_u_def}
\hat{u} \dfn \inf\cbra{x\such \sup_{z\in E}\left(f(z)-x a(z)\right)\leq 0};\qquad \hat{l} \dfn\sup\cbra{x\such \inf_{z\in E}\left(f(z)-xa(z)\right)\geq 0}.
\end{equation}
Assumption \ref{ass: standing} implies $a(z_0)>0$ for some $z_0\in E$ and thus $0\leq \hat{l}\leq \hat{u}\leq \infty$ with $\hat{l}=\hat{u}$ if and only if for some constant
$c$, $f(z) = c a(z)$ for all $z\in E$. In this case, $X=c$ $\prob^z$ almost
surely for all $z\in E$.  With this notation, \cite{MR748853} proves:

\begin{prop}\label{prop: kliemann}(\cite[Section III]{MR748853}) Let Assumptions \ref{ass: standing} hold. Assume that $f\in C^{1,\gamma}(E;\Real_+)$ and $\eta, \theta\equiv 0$. Then the support of $\pi$ is $\bar{E}\times [\hat{l},\hat{u}]$ ($[\hat{l},\infty)$ if $\hat{u}=\infty$).
\end{prop}
\end{rem}

\section{A Numerical Example}\label{S: num_ex}

We now provide an example which highlights the superiority (in terms of computational efficiency) of the time-reversal method over the naive method for obtaining the distribution of $X_0$.  Consider the case $E=\mathbb{R}$, and 
\begin{equation}\label{eq: ou_case}
dZ_t = -\gamma Z_t\ud t + \ud W_t;\qquad X_0 = \int_0^\infty Z_t e^{-at}\ud t;\qquad \gamma,a > 0.
\end{equation}
Note that the function $\Real \ni z \mapsto f(z) = z$ fails to be non-negative. However, as argued below, the results of Theorem \ref{thm:main} still hold. As $Z$ is a mean-reverting Ornstein Uhlenbeck process, it is straight-forward to verify Assumption (A3) with $p(z) = \sqrt{\gamma/\pi}e^{-\gamma z^2}$, so that $Z_0\sim N(0,1/(2\gamma))$.  We claim that $(Z_0,X_0)$ is normally distributed with mean vector $(0,0)$ and covariance matrix
\begin{equation*}
\Sigma = \left(\begin{array}{c c} \frac{1}{2\gamma} &\frac{1}{2\gamma(a+\gamma)}\\ \frac{1}{2\gamma(a+\gamma)} & \frac{1}{2\gamma a(a+\gamma)}\end{array}\right).
\end{equation*}
Indeed, integration by parts shows that for $T>0$:
\begin{equation*}
\int_0^T e^{-at}Z_t \ud t = \frac{Z_0}{a+\gamma} + \frac{1}{a+\gamma}\int_0^T e^{-at} \ud W_t - \frac{1}{a+\gamma}e^{-aT}{Z_T}.
\end{equation*}
The ergodicity of $Z$ implies $\lim_{T \to \infty} \pare{Z_T/T} = -\gamma\int_{\mathbb{R}}zp(z) dz = 0$ almost surely; therefore, it follows that $\lim_{T \to \infty}  e^{-aT}Z_T = 0$ holds almost surely.  Next, note that $Y_T\dfn \int_0^T e^{-at} \ud W_t$ is independent of $Z_0$ and normally distributed with mean $0$ and variance $(1-e^{-2aT})/(2a)$. Lastly, as a process, $Y = \pare{Y_T}_{T\geq 0}$ is an $L^2$-bounded martingale and hence $Y_\infty \dfn \lim_{T \to \infty} Y_T$ almost surely exists, where $Y_\infty$ is independent of $Z_0$, and normally distributed with mean $0$ and variance $1/(2a)$. Thus $X_0 = \lim_{T\uparrow\infty} \int_0^T e^{-at}Z_t dt$ exists almost surely and
\begin{equation*}
X_0 = \frac{Z_0}{a+\gamma} + \frac{Y_\infty}{a+\gamma};\qquad Z_0\indep Y_\infty;\quad  Z_0\sim N\left(0,\frac{1}{2\gamma}\right), Y_\infty\sim N\left(0,\frac{1}{2a}\right),
\end{equation*}
from which the joint distribution follows. Now, even though $f(z)=z$ can take negative values, the time reversal dynamics in \eqref{eq: ze_chi_dyn_deg} still hold, taking the form
\begin{equation*}
\ud \zeta_t = -\gamma\zeta_t \ud t + \ud W_t;\qquad \ud \chi_t = \left(a - \zeta_t\chi_t\right) \ud t.
\end{equation*}
Lastly, even though Theorem \ref{thm:main} no longer directly applies, it is shown in \cite[Theorem 3.3, Section 3.D, Proposition 3.15]{MR748853} that $(\zeta,\chi)$ is still ergodic\footnote{The tightness condition in Proposition 3.15 is straightforward to verify.}, in that \eqref{eq: main_pi_conv} holds.

For these dynamics, we performed the following test: for a fixed terminal time $T$ and mesh size $\delta$, we estimated the distribution of $X_0$ in two ways. First, (``Method A'') by sampling $\zeta_0\sim p$ and setting $\chi_0 = 1$, and second (``Method B'') by running the forward process $Z$ until $2T$ then setting $\zeta_t = Z_{2T-t}$, $\chi_0=1$. For each simulation we computed the empirical distribution along a single path and then estimated the Kolmogorov-Smirnov distance ($d_{KS}(F,G) = \sup_{x}|F(x)-G(x)|$, for distribution functions $F,G$) between the empirical and true distribution for $X_0$. The parameter values were $\gamma = 2$, $a=1$, $T=10,000$ and $\delta = 1/24$. 

\begin{figure}
\epsfig{file=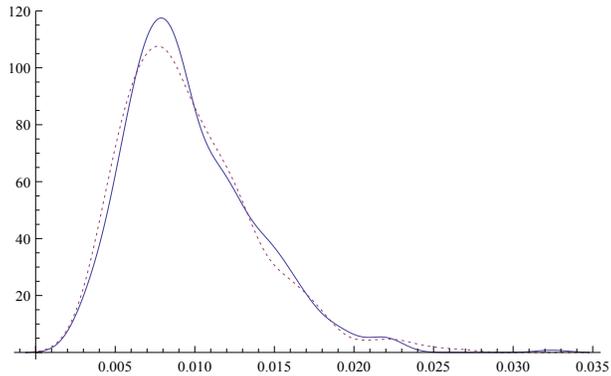,height=5cm,width=8cm}
\caption{Kolmogorov-Smirnov distances between the empirical and true distribution for $X_0$. The solid line is for the reversal method starting $\ze_0\sim p$ and the dashed line for the reversal method running $Z_0$ to $2T$ and setting $\ze_t = Z_{2T-t}$. Here, $T=10,000$, $\delta = 1/24$, $\gamma = 2$ and $a=1$. Computations were performed using \emph{Mathematica} and the code can be found on the author's website \url{www.math.cmu.edu/users/scottrob/research}.} \label{F:KSDist}
\end{figure}

\begin{table}
\centering
\begin{tabular}{l l l}
\hline
 & Method A & Method B \\
\hline
Median Distance & 0.00887 & 0.00882 \\
Standard Deviation & 0.00405 & 0.00413\\
$99^{th}$ Percentile & 0.02168 & 0.02255\\
$1^{st}$ Percentile & 0.00405 & 0.00290\\
Median Time (seconds) & 2.694 & 8.766\\
\hline
\end{tabular}
\caption{Statistics on Kolmogorov-Smirnov distances between the empirical and true distribution for $X_0$ using methods A and B.}  \label{T:summary_data}
\end{table} 

Figure \ref{F:KSDist} shows the resulting Kolmogorov-Smirnov distances for $500$ sample paths. The plot gives a (smoothed) histogram comparing the distances using the two methods described above. As can be seen, the two methods give comparable results: this is not surprising given that rapid convergence of the distribution of $\ze$ to its invariant distribution \cite{MR972776}. Table \ref{T:summary_data} provides summary statistics regarding the median distances and simulation times, as well as the standard deviation and tail data.

Having obtained Kolmogorov-Smirnov distances using reversal methods, we next compared our results to a naive simulation of $X_0$, obtained by sampling $Z_0\sim p$ and computing $X_0$ via \eqref{eq: ou_case} directly. Here, for the median distance $d$ using Method A from Table \ref{T:summary_data}, we sampled $X_0$ stopping at the first instance $N$ so that the Kolmogorov-Smirnov distance between the empirical and true distribution for $X_0$ fell below $d$.  As can be seen from Figure \ref{F:Paths_Naive} and the summary statistics in Table \ref{T:summary_data_naive}, the naive simulation performs significantly worse: at the median it took $7,002$ paths and a simulation time of $8.66$ minutes to achieve the same level of accuracy as $1$ path ($2.94$ seconds) of the reversed process. Further, the histogram shows the presence of a significant number of trials where significantly more than the median number of paths were needed to achieve the given accuracy. 

\begin{figure}
\epsfig{file=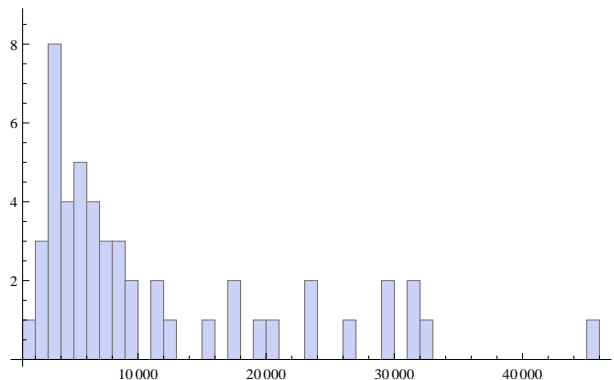,height=5cm,width=8cm}
\caption{Histogram for the number $N$ of paths necessary so that, using the naive simulation for $X_0$, the Kolmogorov-Smirnov distance between the empirical distribution and true distribution for $X_0$ fell below the median distance $d$ using Method A from Table \ref{T:summary_data}. The integral was computed using $T=100$ with mesh size of $\delta =1/24$; furthermore, the values $\gamma=2$ and $a=1$ we used. Computations were performed using \emph{Mathematica} and the code can be found on the author's website \url{www.math.cmu.edu/users/scottrob/research}.} \label{F:Paths_Naive}
\end{figure}

\begin{table}
\centering
\begin{tabular}{l l}
\hline
Summary for the Forward Simulation & \\
\hline
Median Number of Paths & 7,002 \\
Mean Number of Paths & 11,446\\
Standard Deviation & 10,165\\
Minimum Number of Paths & 1,846\\
Maximum Number of Paths & 45,004\\
Median Simulation Time (minutes) & 8.66\\
Mean Simulation Time (minutes) & 14.34\\
\hline
\end{tabular}
\caption{Summary statistics using the naive forward simulation method.}
\label{T:summary_data_naive}
\end{table}

\section{Conclusion} 

In this work, using the method of time reversal, an efficient method for simulating the joint distribution of $(Z_0,X_0)$ for perpetuities of the form \eqref{eq: X_def} is obtained. The joint distribution may be obtained by sampling a single path of the reversed process, as opposed to sampling numerous paths of $X_0$ using the naive method. However, the effectiveness of the proposed method depends on being to obtain analytic representations for the distribution $p$ of $Z_0$: an undertaking that, though always possible in the one-dimensional case, is often not possible for non-reversing multi-dimensional diffusions. Furthermore, results are presented for perpetuities with non-negative underlying cash flow rates. As such, more research is needed to identify an effective time reversal method for perpetuities of the form
\begin{equation*}
X_0 = \int_0^\infty D_t dF_t
\end{equation*}
for general Markovian processes $F$ (i.e., not just $dF_t = f(Z_t)dt$) containing both jumps and diffusive terms. Additionally, the performance of the method where $Z$ is run until a large time $2T$ and then setting $\ze_t = Z_{2T-t}$ must be thoroughly analyzed: in particular, how fast does the distribution of $Z_{2T}$ converge to $p$ given a fixed starting point?  To answer these questions, one must first analyze the resultant backwards dynamics and associated PDEs for the invariant density, obtaining rates of convergence. % Lastly, the effectiveness of the proposed simulation methodology is also improved when one is able ascertain the existence of a density for $X_0$; therefore, it is natural to question when is it the case that a density for $X_0$ exists. Aside from being a question of independent interest, a positive answer would support the use of time reversal to efficiently obtain the joint distribution.

\section{Proofs from Section \ref{subsec: X_finite}}\label{sec: X0_finite}

We present here the proofs of Lemma \ref{lem: X_finite_A} and Lemma \ref{lem: X_infinite_A}.

\begin{proof}[Proof of Lemma \ref{lem: X_finite_A}]
Let $\eps > 0$ be as in \eqref{eq: eps_finite_cond}. Assume first that $\theta'c\theta + \eta'\eta \equiv 0$. Then $R = \int_0^\cdot a(Z_t)\ud t$ and \eqref{eq: eps_finite_cond} reads $a^- \in \Lb^1(E,p)$ and $\int_E a(z) p(z) \ud z > 0$. Set $\kappa \dfn(1/4) \int_E a(z) p(z) \ud z > 0$. Fix $z\in E$ and denote by $\prob^z$ the probability obtained by conditioning upon $Z_0 = z$. The positive recurrence of $Z$ implies (\cite[Theorem 4.9.5]{MR1326606}), there exists a $\prob^z$-a.s. finite random variable $T(z)$ such that $t \geq T(z)$ implies that $R_t = \int_0^t a(Z_u)\ud u \geq 2 \kappa t $ and hence the first conclusion of Lemma \ref{lem: X_finite_A} holds. Furthermore, since $Z$ is stationary, ergodic under $\prob$, the ergodic theorem implies there is a $\prob$ a.s. finite random variable $T$ such that $t\geq T$ implies $R_t \geq 2\kappa t$.  Now, let $n\in\mathbb{N}$ be such that $n > 1/(2\kappa)$. We have
\begin{equation*}
\sup_{t\geq 0} (t/n - R_t) \leq \sup_{t\leq T}(t/n-R_t) < \infty,
\end{equation*}
where the last inequality follows by the regularity of $a$ and the non-explositivity of $Z$. Thus
\begin{equation*}
X_0  = \int_0^\infty e^{-R_t}f(Z_t)\ud t \leq e^{\sup_{t\leq T}(t/n-R_t)}\int_0^\infty e^{-t/n}f(Z_t)\ud t.
\end{equation*}
By the stationarity of $Z$:
\begin{equation*}
\expec\bra{\int_0^\infty e^{-t/n}f(Z_t)\ud t} = \int_0^\infty e^{-t/n}\expec\bra{f(Z_t)}\ud t = n \int_E f(z)p(z)dz < \infty,
\end{equation*}
hence $\prob\bra{\int_0^\infty e^{-t/n}f(Z_t)dt < \infty} = 1$, which in turn implies that $\prob\bra{X_0 < \infty} = 1$.

%---begin commenting out ----%

%For this $T$,
%\begin{equation*}
%X_0 = \int_0^\infty f(Z_t)e^{-R_t}\ud t \leq T\sup_{t \in [0, T]} \pare{ f(Z_t)e^{-R_t} } + \int_0^\infty f(Z_t)e^{- \kappa t}\ud t.
%\end{equation*}
%Since $f\in \locbddmbl$ and since for each $\omega$, non-explosion of $Z$ implies the path $\cbra{Z_t(\omega)}_{t\leq T(\omega)}$ is contained with some $E_n(\omega)$ it follows that the first term on the right hand side above is finite almost surely.  As for the second term, by the stationarity of $Z$,
%\begin{equation*}
%\expec\bra{\int_0^\infty f(Z_t)e^{-\kappa t }\ud t} = \frac{1}{\kappa}\int_E f(z) p(z) \ud z < \infty.
%\end{equation*}
%Therefore, $\prob \bra{X_0 < \infty} = 1$ holds.

%---end commenting out -----%

Assume now that $\theta'c\theta + \eta'\eta \not\equiv 0$, which by continuity of all involved functions implies that $\int_E \pare{\theta'c\theta + \eta'\eta} (z) p(z) \ud z > 0$. Fix $z\in E$. Positive recurrence of $Z$ gives that $\lim_{t\to\infty} \int_0^t(\theta'c\theta+\eta'\eta)(Z_u)\ud u = \infty$ with $\prob^z$ probability one. On the event $\big \{ \int_0^t(\theta'c\theta +
\eta'\eta)(Z_u)\ud u > 0 \big \}$, note that
\begin{equation*}
-R_t = -\int_0^t a(Z_u)\ud u + \int_0^t(\theta'c\theta +
\eta'\eta)(Z_u)\ud u\left(-\frac{1}{2} -\frac{\int_0^t\theta'\sigma(Z_u)\ud W_u +
    \eta(Z_u)\ud B_u}{\int_0^t(\theta'c\theta+\eta'\eta)(Z_u)\ud u}\right).
\end{equation*}
By the Dambis, Dubins and Schwarz theorem and the strong law of large numbers for Brownian motion, it follows that there exists a $\prob^z$-a.s. finite random variable $T(z)$ such that
\begin{equation*}
t \geq T(z) \quad \Longrightarrow \quad -\frac{\int_0^t\theta'\sigma(Z_u)\ud W_u +
    \eta(Z_u)\ud B_u}{\int_0^t(\theta'c\theta+\eta'\eta)(Z_u)\ud u} \leq
  \frac{\eps}{2};
\end{equation*}
therefore,
\begin{equation*}
t \geq T(z) \quad \Longrightarrow \quad  -R_t \leq -\int_0^t\left(a + \frac{1-\eps}{2}(\theta'c\theta +
  \eta'\eta)\right)(Z_u)\ud u.
\end{equation*}
With $\kappa \dfn (1/4)\int_E ( a + (1-\eps)(\theta'c\theta+\eta'\eta) /2)(z) p(z) \ud z > 0$, and increasing $T(z)$ if necessary (still keeping it $\prob^z$-a.s. finite), it follows that $t\geq T(z)$ implies $-R_t \leq - 2 \kappa
t$. Hence the first part of Lemma \ref{lem: X_finite_A} holds true again. Additionally, the ergodic theorem applied with $\prob$ gives a $\prob$-a.s. finite random variable $T$ such that $t\geq T$ implies $-R_t \leq -2\kappa t$.  Again, for $n\in\mathbb{N}$ such that $n > 1/(2\kappa)$ we have
\begin{equation*}
X_0 = \int_0^\infty e^{-R_t}f(Z_t)\ud t \leq e^{\sup_{t\leq T}(t/n-R_t)}\int_0^\infty e^{-t/n}f(Z_t)\ud t.
\end{equation*}
from which $\prob \bra{X_0 < \infty} = 1$ follows by the same line of reasoning as above.
\end{proof}

\begin{proof}[Proof of Lemma \ref{lem: X_infinite_A}]
The proof is nearly identical that if Lemma \ref{lem: X_finite_A}. Namely,
in each of the cases $\theta'c\theta + \eta'\eta\equiv 0$ and
$\theta'c\theta + \eta'\eta\not\equiv 0$, under the given hypothesis there
is a constant $\kappa \geq0$ and a $\prob$-a.s. finite random variable $T$ such that $-R_t \geq \kappa t$ holds for $t\geq T$. This gives that
\begin{equation}\label{eq: X_infinite_A1}
X_0 \geq \int_T^\infty e^{\kappa t}f(Z_t)\ud t \geq \int_T^\infty e^{\kappa t}(f\wedge N)(Z_t)dt,
\end{equation}
where $N$ is large enough so that $\int_E (f(z) \wedge N) p(z) \ud z > 0$.  We thus have
\begin{equation*}
X_0 \geq \int_0^\infty e^{\kappa t}(f\wedge N)(Z_t)dt - \frac{N}{\kappa}(e^{\kappa T}-1).
\end{equation*}
Ergodicity of $Z$ implies that $\prob$ almost surely
\begin{equation*}
\lim_{u\to\infty}\frac{1}{u}\int_0^u(f\wedge N)(Z_t)\ud t = \int_E \pare{f(z)\wedge N} p(z) \ud z > 0,
\end{equation*}
so that $\lim_{u\to\infty} \int_0^ue^{\kappa t} (f\wedge N)(Z_t)\ud t = \infty$, proving the result.
\end{proof}

\section{Proof of Theorem \ref{thm: pde_soln}} \label{sec: proof_of_pde}

Under the given assumptions there exists a unique solution, $(\prob^{z,x})_{(z,x)\in F}$ to the generalized martingale problem for $L$ on $F$, where $L$ is from \eqref{eq: L_def}. Here, the measure space is $(\tOmega,\tbF)$, where $\tOmega = (C[0,\infty);
\hat{F})$, with $\hat{F}$ being the one-point compactification of $F$. The filtration $\tbF$ is the right-continuous enlargement of the filtration generated by the coordinate process $(\widetilde{Z}, \widetilde{Y})$ on $\tOmega$.

Let $(F_n)_{\nin}$ be an increasing sequence smooth, bounded, open, connected domains of $F$ such that $F = \cup_{n} F_n$. Note that $F_n$ can be obtained by smoothing out the boundary of $E_n\times (1/n,n)$. By uniqueness of solutions to the generalized martingale problem, for each $n$, the law of of $(\widetilde{Z},\widetilde{Y})$ is the same as the law of $(Z,Y^x)$ under $\prob\bra{\cdot\such Z_0 =z}$ (where the latter will always denote a version of the conditional probability) up until the first exit time of $F_n$. Furthermore, since the process $Z$ is recurrent, with $(\prob^z)_{z\in E}$ being the restriction of $(\prob^{z,x})_{(z,x)\in F}$ to the first $d$ coordinates, for $z \in E$, the law of $\widetilde{Z}$ under $\prob^z$ is the same as the law of $Z$ under $\prob \bra{\, \cdot \such Z_0 = z}$. For these reasons, and in order to ease the reading, we abuse notation and still use $(Z, Y)$ instead of $(\widetilde{Z}, \widetilde{Y})$ for the coordinate process on $\tOmega$. The underlying space we are working on will be clear from the context.

Denote by $\tau_n$ the first exit time of $(Z,Y)$ from $F_n$.  Assumption \ref{ass: standing} implies $Z$ does not explode under $\prob^{z,x}$ and $Y$ cannot explode to infinity since $D$ is strictly positive almost surely under $\prob \bra{\, \cdot \such Z_0 = z}$ for all $z \in E$. Therefore, the explosion time $\tau\dfn \lim_{n\to\infty} \tau_n$ for $(Z,Y)$ is the first hitting time of $Y$ to $0$ and the law of $\tau$ under $\prob^{z,x}$ is the same as the law of the first hitting of $Y^x$ to $0$ under $\prob \bra{\, \cdot \such Z_0 = z}$.

Note that $Y^x_t = D_t^{-1}\left(x-X_0 + \int_t^\infty D_u
  f(Z_u)\ud u\right)$. Assumption \ref{ass: standing} implies\footnote{This follows by the ergodic theorem since
  $\cbra{\int_t^\infty f(Z_u)D_u\ud u = 0} \subset \cbra{\lim_{k\to\infty}
    (1/k)\int_t^{t+k}f(Z_u)\ud u = 0}$.}
\begin{equation}\label{eq: Xt_Pz_pos}
\prob \bra{\int_t^\infty D_uf(Z_u)\ud u > 0 \ \Big| \ Z_0 = z} = 1, \qquad z\in E, \quad t\geq
0.
\end{equation}
Therefore,
\begin{equation*}
g(z,x) = \prob \bra{X_0\leq x \such Z_0 = z} = \prob^{z, x} \bra{Y^x_t > 0, \
  \forall t\geq 0} = \prob^{z,x}\bra{\tau =
  \infty}.
\end{equation*}
Define
\begin{equation}\label{eq: h_def}
h(z,x) \dfn \prob^{z,x}\bra{\lim_{t\to\infty} Y_t =
  \infty}, \quad (z,x) \in F
\end{equation}
Fix $(z,x)\in F$ and let $0 < \eps < x$. Note that $Y^{x}_t = Y^{x-\eps}_t + \eps / D_t$. Since $\lim_{t \to \infty} D_t = 0$ holds $\prob \bra{ \cdot \such Z_0 = z}$-a.s., it follows that
\begin{equation}\label{eq: tempp2}
\begin{split}
\prob^{z,x-\eps}\bra{\tau = \infty} &=\prob \bra{Y^{x-\eps}_t > 0\
  \forall t\geq 0 \such Z_0 = z} \\
  &\leq \prob\bra{Y^x_t \geq  \eps / D_t, \ \forall t\geq 0 \such Z_0 = z}\\
  &\leq \prob\bra{\lim_{t \to \infty}Y^x_t = \infty \such Z_0 = z} \\
&= \prob^{z,x}\bra{\lim_{t\to\infty}Y_t = \infty} \leq \prob^{z,x}\bra{\tau = \infty}.
\end{split}
\end{equation}
Therefore, $g(z,x-\eps)\leq h(z,x) \leq g(z,x)$. By definition, $g(z,x)$ is
right-continuous in $x$ for a fixed $z$ and so
\begin{equation*}
g(z,x) \leq \liminf_{\eps\to 0} h(z,x+\eps) \leq
\limsup_{\eps\to 0}h(z,x+\eps) \leq \limsup_{\eps\to 0}
g(z,x+\eps) = g(z,x).
\end{equation*}
Therefore, if $h(z,x)$ is continuous it follows that $h(z,x) = g(z,x)$.  It is
now shown that in fact $h$ is in $C^{2,\gamma}(F)$ and satisfies
$Lh = 0$. This gives the desired result for $g$ since $g = h$.

Let $\psi:(0,\infty)\mapsto (0,1)$ be a smooth function such that
$\lim_{x\to 0}\psi(x) = 0$, $\lim_{x\to\infty}\psi(x) = 1$. By the classical Feynman-Kac formula
\begin{equation*}
u^n(z,x)\dfn \expec^{\prob^{z,x}}\bra{\psi(Y_{\tau_n})},
\end{equation*}
satisfies $Lu^n = 0$ in $F_n$ with $u^n(z,x) = \psi(x)$ on
$\partial F_n$.  Since $\prob \bra{X_0 < \infty \such Z_0 = z} = 1$ there exists a
pair $(z_0,x_0) \in F$ so that $\prob \bra{X_0 < x_0 \such Z_0 = z_0} > 0$.  Using \eqref{eq: tempp2} this gives
\begin{equation}\label{eq: tempp1}
h(z_0,x_0) \geq \prob\bra{X_0 < x_0\such Z_0=z_0} > 0.
\end{equation}
Therefore, $(\prob^{z,x})_{(z,x)\in F}$ is transient \cite[Chapter
2]{MR1326606} and, since $(\prob^z)_{z\in E}$ is positive recurrent, this implies that for all $(z,x)$, with
$\prob^{z,x}$-probability one, either $\lim_{t\to\tau}Y_t = 0$ or
$\lim_{t\to\tau} Y_t = \infty$, where in the latter case, $\tau =
\infty$  since $Y$ cannot explode to $\infty$. This in turn yields that
$Y_{\tau_n}\rightarrow 0$ or $Y_{\tau_n}\rightarrow \infty$ with
$\prob^{z,x}$-probability one and hence by the dominated convergence
theorem
\begin{equation}\label{eq: tempp0}
\lim_{n\to\infty} u^n(z,x) = \prob^{z,x}\bra{\lim_{t\to\tau}
  Y_t = \infty} = \prob^{z,x}\bra{\lim_{t\to\infty}Y_t  =\infty} = h(z,x).
\end{equation}
For $(z_0,x_0)$ from \eqref{eq: tempp1}, $g(z_0,x_0)\geq h(z_0,x_0) > 0$ and hence $g(z,x) > 0$ for all $(z,x)\in F$ \cite[Theorem
1.15.1]{MR1326606}. But this implies $h(z,x) \geq g(z,x/2) > 0$, and so from \eqref{eq: tempp0} the $u^n$ are converging point-wise to a
strictly positive function. Thus, by the interior Schauder estimates and
Harnack's inequality, it follows by ``the standard
compactness'' argument (\cite[Page 147]{MR1326606}) that there exists a $C^{2,\gamma}(F)$,
strictly positive, function $u$ such that $u^n$ converges to $u$ in the
$C^{2,\gamma}(D)$ Holder space for all compact $D\subset F$.
Clearly, this function $u$ satisfies $Lu = 0$ in $F$.
In fact, since $u^n$ converges to $h$ pointwise, $h=u$ and hence $Lh = 0$.

The boundary conditions for $g$ are now considered. Let the integer $k$ be
given. It suffices to show for each $\eps > 0$ there is some $n(\eps)$ such that
\begin{equation}\label{eq: g_bc_eps}
\sup_{x\leq n(\eps)^{-1},z\in E_{k}}g(z,x)\leq \eps;\qquad \inf_{x\geq n(\eps),z\in
  E_{k}}g(z,x)\geq 1-\eps
\end{equation}
The condition near $x = 0$ is handled first. By way of contradiction, assume
there exists some $\eps > 0$ such that for all integers $n$
there exists $z_n\in E_{k}$, $x_n \leq 1/n$ such that $g(z_n,x_n) >
\eps$.  Since the $z_n$ are all contained within $E_k$ there is a sub-sequence
(still labeled $n$) such that $z_n\rightarrow z$ for $z\in\bar{E}_k$.

Let $\delta > 0$ and choose $N_\delta$ such that $n\geq N_{\delta}$ implies
  $n^{-1}\leq\delta$. Since $g$ is increasing in $x$,  $\eps <
g(z_n,\delta)$.  Since $g$ is continuous, $\eps\leq g(z,\delta)$.  Since
this is true for all $\delta > 0$, $\lim_{x\to 0} g(z,x)\geq
\eps$. But, this is a contradiction : $\lim_{x\to 0}g(z,x) = 0$ for
each $z\in E$.  To see this,
let $\delta > 0$ and choose $\beta > 0$ such that
$\prob \bra{X_0 \geq \beta \such Z_0 = z} \geq 1-\delta$.  This is possible in view of \eqref{eq: Xt_Pz_pos}. Thus, for $x < \beta$, $g(z,x) \leq \prob \bra{X_0 < \beta \such Z_0 = z} \leq \delta$ and hence $\limsup_{x\to 0} g(z,x) \leq \delta$. Taking
$\delta\to 0$ gives the result.

The proof for $x\to\infty$ is very similar. Assume by contradiction that
there is some $\eps > 0$ such that for all
integers $n$ there exist $z_n\in E_{k}$, $x_n \geq n$ such that $g(z_n,x_n) <
1-\eps$. Again, by taking sub-sequences, it is possible to assume
$z_n\rightarrow z\in\bar{E}_k$. Fix $M > 0$.  For $n\geq M$, since $g$ is increasing in
$x$,  $g(z_n,M) < 1-\eps$.  Since $g$ is continuous, $g(z,M)\leq 1-\eps$.
Since this holds for all $M$, $\lim_{x\to\infty}g(z,x) \leq 1-\eps$. But, this violates the condition that under $\prob \bra{\cdot \such Z_0 = z}$, $X_0<\infty$ almost surely.

The uniqueness claim is now proved. Let $\tilde{g}$ be a $C^2(F)$
solution of $L\tilde{g} = 0$ such that $0\leq \tilde{g}\leq 1$ and such that
\eqref{eq: g_bc} holds. Define the stopping times
\begin{equation}\label{eq: stopping_times}
\sigma_k \dfn \inf\cbra{t\geq 0 : Z_t\not\in E_k};\qquad \rho_k \dfn
\inf\cbra{t\geq 0 : Y_t = k}.
\end{equation}
By Ito's formula, for any $k,n,m$
\begin{equation*}
\begin{split}
\tilde{g}(z,x) &= \expec^{\prob^{z,x}}\bra{g(Z_{\sigma_k\wedge\rho_{1/n}\wedge\rho_m},
  Y_{\sigma_k\wedge\rho_{1/n}\wedge\rho_m})\left(1_{\rho_{1/n} <
      \sigma_k\wedge \rho_m} + 1_{\rho_{1/n}\geq
      \sigma_k\wedge \rho_m}(1_{\tau <\infty} + 1_{\tau =\infty})\right)}.
\end{split}
\end{equation*}
Since $\prob^{z,x}$ almost surely $\lim_{m\to\infty} \rho_m = \infty$, taking $m\to\infty$ yields
\begin{equation*}
\begin{split}
\tilde{g}(z,x) &= \hat{\expec}^{\prob^{z,x}}\bra{g(Z_{\sigma_k\wedge\rho_{1/n}},
  Y_{\sigma_k\wedge\rho_{1/n}})\left(1_{\rho_{1/n}<
      \sigma_k} + 1_{\rho_{1/n}\geq\sigma_k}(1_{\tau <\infty} + 1_{\tau=\infty})\right)}.
\end{split}
\end{equation*}
On $\cbra{\rho_{1/n} < \sigma_k}$, $Z_{\rho_{1/n}}\in E_k$, $Y_{\rho_{1/n}}\leq 1/n$ and hence by $0\leq
\tilde{g}\leq 1$ and \eqref{eq: g_bc}, for any $\eps > 0$ there is an $n(\eps)$
such that for $n\geq n(\eps)$
\begin{equation*}
\begin{split}
\tilde{g}(z,x) &\leq \eps +  \prob^{z,x}\bra{\rho_{1/n}\geq\sigma_k,\tau<\infty} + \prob^{z,x}\bra{\rho_{1/n}\geq\sigma_k, \tau=\infty}.
\end{split}
\end{equation*}
Taking $n\to\infty$ thus gives
\begin{equation*}
\begin{split}
\tilde{g}(z,x) &\leq \eps +  \prob^{z,x}\bra{\tau \geq\sigma_k,\tau<\infty} + \prob^{z,x}\bra{\tau=\infty}.
\end{split}
\end{equation*}
Taking $k\to\infty$ gives
\begin{equation*}
\begin{split}
\tilde{g}(z,x) &\leq \eps + \prob^{z,x}\bra{\tau=\infty}.
\end{split}
\end{equation*}
and hence taking $\eps\to 0$ gives $\tilde{g}(z,x) \leq
\prob^{z,x}\bra{\tau = \infty} = g(z,x)$.  Similarly, for $k,n,m$
\begin{equation*}
\begin{split}
\tilde{g}(z,x) &= \expec^{\prob^{z,x}}\bra{g(Z_{\sigma_k\wedge\rho_{1/n}\wedge\rho_m},
  Y_{\sigma_k\wedge\rho_{1/n}\wedge\rho_m})\left(1_{\rho_{m}<
      \sigma_k\wedge\rho_{1/n}} + 1_{\rho_{m}\geq
      \sigma_k\wedge\rho_{1/n}}\right)},\\
&\geq (1-\eps)\hat{\prob}^{z,x}\bra{\rho_{m}<
      \sigma_k\wedge \rho_{1/n},\lim_{t\to\infty}Y_t = \infty},
\end{split}
\end{equation*}
for all $\eps > 0$ and $m\geq m(\eps)$ for some $m(\eps)$.  Note that the set
$\cbra{\rho_m < \sigma_k\wedge\rho_{1/n}}$ is restricted to include
$\cbra{\lim_{t\to\infty} Y_t = \infty}$ but this is fine since lower bounds are considered. Now, on the event $\cbra{\lim_{t\to\infty} Y_t = \infty}$ it holds that $\rho_{1/n}\rightarrow \infty$.
Thus, taking $n\to\infty$
\begin{equation*}
\tilde{g}(z,x)\geq (1-\eps)\prob^{z,x}\bra{\rho_m < \sigma_k, \lim_{t\to\infty}Y_t =
  \infty}.
\end{equation*}
Taking $k\to\infty$ gives
\begin{equation*}
\tilde{g}(z,x)\geq (1-\eps)\prob^{z,x}\bra{\rho_m < \infty,  \lim_{t\to\infty}Y_t =
  \infty}.
\end{equation*}
Taking $m\to\infty$ and noting that for $m$ large enough $\rho_m < \infty$ on $\lim_{t\to\infty}Y_t =
  \infty$ it holds that
\begin{equation*}
\tilde{g}(z,x) \geq (1-\eps)\prob^{z,x}\bra{\lim_{t\to\infty}Y_t=\infty} = (1-\eps) h(z,x).
\end{equation*}
where the last equality follows by the definition of $h$ in \eqref{eq: h_def}.  Now, in proving $Lg = 0$ it was shown that $g = h$ and hence $\tilde{g}(z,x)\geq (1-\eps)g(z,x)$. Taking $\eps\to 0$ gives that $\tilde{g}(z,x) \geq g(z,x)$, finishing the proof.

\section{Dynamics for the Time-Reversed Process}\label{sec: tr}

The goal of the next two sections is to prove Theorem \ref{thm:main}. We keep all notation from Subsection \ref{subsec: time_reverse}. We first identify the dynamics for $\ze^T$.

\begin{prop}\label{prop: zeta_t_dynamics}

Let Assumptions \ref{ass: standing} hold.  Then, for each $T > 0$, the law of $\ze^T$ under $\prob$ solves the martingale problem on $E$ (for $t\leq T$) for the operator $L^\ze \dfn (1/2)c^{ij}\partial^2_{ij} +
\mu^i\partial_i$ where
\begin{equation}\label{eq: mu_def}
\mu \dfn c\frac{\nabla p}{p} + \dvg{c} - m.
\end{equation}
The operator $L^{\ze}$ does not depend upon $T$. Thus, if $\left(\qprob^{z}\right)_{z\in E}$ denotes the solution of the generalized
  martingale problem for $L^\ze$ on $E$, then
  in fact $(\qprob^{z})_{\ze\in E}$ solves the martingale problem for
  $L^{\ze}$ on $E$ and is positive recurrent.
\end{prop}

\begin{rem} If $Z$ is reversing then $p$ satisfies $m = (1/2)\left(c\nabla{p}/p + \dvg{c}\right)$. Thus, in this instance, $\mu = m$ and, as the name suggests, $\ze^T$ has the same dynamics as $Z$.

\end{rem}

\begin{proof}
The first statement regarding the martingale problem is based off the argument in \cite{MR866342}. Since $Z$ is positive recurrent with invariant measure $p$ and  $Z_0$ has initial distribution $p$ under $\prob$, $Z$ is stationary with distribution $p$. Since $\tilde{L}^Zp = 0$, equation $(2.5)$ in \cite{MR866342} holds noting that $p$ does not depend upon $t$.

For a given $s\leq t \leq t$ and $g\in C^{\infty}_c(E)$ define the function $v(s,z) \dfn \expec\bra{g(X_t)\big | Z_s = z}$.  The Feynman-Kac formula implies $v$ satisfies $v_s + L^zv = 0$ on $0 <  s < t, z \in E$ with $v(t,z) = g(z)$ : see \cite{Heath-Schweizer,MR3146488} for an extension of the classical Feynman-Kac formula to the current setup. Therefore, the condition in equation $(2.7)$ of \cite{MR866342} holds as well. Thus, the formal argument on page $1191$ of \cite{MR866342} is rigorous and the law of $\ze^T$ under $\prob$ solves the martingale problem for $L^\ze$.

Turning to the statement regarding $\left(\qprob^z\right)_{z\in E}$, set $\tilde{L}^{\ze}$ as the formal adjoint to $L^{\ze}$.  $\tilde{L}^{\ze}$ is given by \eqref{eq: tildeLZ} with $\mu$ replacing $m$ therein.  Using the formula for $\mu$ in \eqref{eq: mu_def} and for $\tilde{L}^Z$ in \eqref{eq: tildeLZ} calculation shows that
\begin{equation*}
\tilde{L}^{\ze}f = \tilde{L}^Z f - 2\nabla\cdot\left(\frac{f}{p}\left(\frac{1}{2}\left(c\nabla p + p\dvg{c}\right)-pm\right)\right).
\end{equation*}
Since
\begin{equation}\label{eq: p_density_pde}
0 = \tilde{L}^Zp = \nabla\cdot\left(\frac{1}{2}\left(c\nabla p + p\dvg{c}\right) - pm\right),
\end{equation}
it follows by considering $f = p$ above that $\tilde{L}^{\ze}p = 0$. Therefore, $p$
is an invariant density for $L^\ze$ if an only if the diffusion corresponding
to the operator $\tilde{L}^{\ze,p}$ does not explode, where
$\tilde{L}^{\ze,p}$ is the h-transform of $\tilde{L}^{\ze}$ \cite[Theorem
4.8.5]{MR1326606}.  But, by definition of the h-transform \cite[pp. 126]{MR1326606} and \eqref{eq: tildeLZ} with $\mu$ replacing $m$:
\begin{equation*}
\begin{split}
\tilde{L}^{\ze,p}f &\dfn \frac{1}{p}\tilde{L}^{\ze}(fp) = \frac{1}{2}c^{ij}\partial^2_{ij}f -\left(\mu^i - \dvg{c}^i-\left(c\frac{\nabla p}{p}\right)^i\right)\partial_i f + \frac{f}{p}\tilde{L}^{\ze}p,\\
&= \frac{1}{2}c^{ij}\partial^2_{ij}f + m^i\partial_i f = L^Zf,
\end{split}
\end{equation*}
where the third equality follows from \eqref{eq: mu_def}. Thus, Assumption \ref{ass: standing} (specifically the fact that $Z$ is ergodic and $\int_E p(z)dz = 1$) implies the diffusion for $\tilde{L}^{\ze,p}$ not only does not explode but also is positive recurrent, finishing the proof.
\end{proof}

In preparation for the proof of the main result of this Section, which is Proposition \ref{prop: zeta_chi_t_dynamics},
it is first needed to define a certain ``backwards'' filtration $\bG^T$ and to present two Lemmas. Fix $T \in (0, \infty)$ and $t \in [0, T]$ and let $\widetilde{\G}^T_t$ be the $\sigma$-field generated by $X_T$, $(Z_{T-u})_{u\in [0,t]}$, $(W_T - W_{T - u})_{u \in [0, t]}$ and $(B_T - B_{T - u})_{u \in [0, t]}$. Then, let $\bG^T \dfn (\G^T_t)_{t \in [0, T]}$ be the usual augmentation of $(\widetilde{\G}^T_t)_{t \in [0, T]}$. It is easy to check that $(\chi^T, \ze^T)$ is $\bG^T$-adapted for all $T \in \Real_+$, as well as that the process $B^T$ defined via $B^T_t \dfn B_{T - t} - B_{T}$ is a $k$ dimensional Brownian motion on $(\Omega, \bG^T, \prob)$, independent of $(\chi^T_0, \ze^T_0)= (X_T, Z_T)$. However, the $\bG^T$-adapted process $(W_{T - t} - W_{T})_{t \in [0,
  T]}$ is not necessarily a Brownian motion on $(\Omega, \bG^T, \prob)$.

With this notation, the following two Lemmas are essential for proving
Proposition \ref{prop: zeta_chi_t_dynamics}.

\begin{lem} \label{lem: int_time_rev}
Let Assumptions \ref{ass: standing} hold. For locally bounded Borel function $\eta: E \mapsto \Real$ and $0 \leq s \leq t \leq T$, it holds that
\begin{equation} \label{eq: int_wrt_B}
- \int_{T - t}^{T - s} \eta(Z_u)' \ud B_u = \int_{s}^{t} \eta(\zeta^T_u)' \ud B^T_u.
\end{equation}
Furthermore, if $\theta: E \mapsto \Real^d$ is continuously differentiable, then
\begin{equation} \label{eq: int_wrt_W}
- \int_{T - t}^{T - s} \theta'(Z_u)\ud Z_u  = \int_{s}^{t} \theta'(\zeta^T_u)\ud \zeta^T_u + \int_s^t\left(\nabla\cdot(c\theta)-\theta'\dvg{c}\right)(\zeta^T_u) \ud u.
\end{equation}
\end{lem}

\begin{proof}
Fix $0 \leq s \leq t \leq T$. For each $\nin$ and $i \in \set{0 ,\ldots, n}$,
let
\begin{equation}\label{eq: uni_def}
u^n_i \dfn T - t + i(t-s)/n.
\end{equation}
First, assume that $\eta$ is twice continuously differentiable. The standard convergence theorem for stochastic integrals implies that (the following limit is to be understood in measure $\prob$):
\[
\int_{s}^{t} \eta(\zeta^T_u)' \ud B^T_u + \int_{T - t}^{T - s} \eta(Z_u)' \ud B_u = - \limn \pare{ \sum_{i=1}^n \pare{ \eta(Z_{u^n_{i}}) - \eta(Z_{u^n_{i-1}}) }' \pare{B_{u^n_{i}} - B_{u^n_{i-1}}} }.
\]
Since $B$ and $Z$ are independent, by Ito's formula the last quadratic covariation is zero. Therefore, \eqref{eq: int_wrt_B} holds for twice continuously differentiable $\eta$. The fact that \eqref{eq: int_wrt_B} holds whenever $\eta$ is locally bounded follows from a monotone class argument.

In a similar manner, assume that $\theta$ is twice continuously differentiable. The standard convergence theorem for stochastic integrals implies that
\begin{align*}
    \int_{s}^{t} \theta'(\zeta^T_u)\ud \zeta^T_u &+ \int_{T - t}^{T
      - s} \theta'(Z_u)\ud Z_u = - \limn \pare{ \sum_{i=1}^n \pare{\theta(Z_{u^n_{i}}) - \theta(Z_{u^n_{i-1}}) }'(Z_{u^n_{i}} - Z_{u^n_{i-1}}) }.
\end{align*}
The last quadratic covariation process (without the minus sign) is equal to
\[
\int_{T-t}^{T - s} \tilde{F}(c,\theta)(Z_u) \ud u =  \int_s^t \tilde{F}(c,\theta)(\zeta^T_u) \ud u,
\]
where $\tilde{F}(c,\theta): E \mapsto\Real$ is given by
\begin{equation*}
\tilde{F}(c,\theta) = \sum_{i,j=1}^{d}c^{ij}\partial_{z_i}\theta^j =
\sum_{i,j=1}^{d}\left(\partial_{z_i}(c^{ij}\theta^j)-\theta^j\partial_{z_i}((c')^{ji})\right)
= \nabla\cdot(c\theta) - \theta'\dvg{c},
\end{equation*}
since $c'=c$. Thus, \eqref{eq: int_wrt_W} is established in the case where $\theta$ is twice continuously differentiable. The fact that \eqref{eq: int_wrt_W} holds whenever $\theta$ is continuously differentiable follows form a density argument, noting that there exists a sequence $(\theta_n)_{\nin}$ of polynomials such that $\limn \theta_n = \theta$ and $\limn \nabla \theta_n = \nabla \theta$ both hold, where the convergence is uniform on compact subsets of $E$.
\end{proof}

\begin{lem} \label{lem: DeltaT_dyn}
Let Assumptions \ref{ass: standing} hold. For each $T \in \Real_+$, define the $\bG^T$-adapted continuous-path $\Delta^T$ as in \eqref{eq: Delta_def}. Then $\Delta^T$ is a semimartingale on $(\Omega, \bG^T, \prob)$. More
precisely, for $t\in[0,T]$
\begin{equation}\label{eq: Delta_dyn}
\begin{split}
\Delta^T_t  &= 1 + \int_0^{t}\Delta^T_u\pare{\theta'c\frac{\nabla p}{p} + \nabla\cdot (c\theta) - a}(\ze^T_u) \ud u\\
&\qquad\qquad + \int_0^t
\Delta^T_u \pare{\eta(\ze^T_u)'\ud B^T_u + \theta'\sigma(\ze^T_u)\ud W^T_u}.\end{split}
\end{equation}
\end{lem}

\begin{proof}
Define $(\rho^T_t)_{t \in [0, T]}$ by $\rho^T_t \dfn R_T - R_{T - t}$, for $t
\in [0, T]$. In view of \eqref{eq: Z_defn}, \eqref{eq: R_defn}, \eqref{eq: mu_def} and Lemma \ref{lem: int_time_rev},
\begin{align*}
\rho^T &= \int_{T - \cdot}^T \pare{a + \frac{1}{2}(\theta'c\theta + \eta'\eta) }(Z_t) \ud t + \int_{T - \cdot}^T \pare{\eta(Z_t)'\ud B_t + \theta'\sigma(Z_t)\ud W_t}  \\
&= \int_{T - \cdot}^T \pare{a-\theta'm + \frac{1}{2}(\theta'c\theta + \eta'\eta) }(Z_t) \ud t + \int_{T - \cdot}^T \pare{\eta(Z_t)'\ud B_t + \theta'(Z_t)\ud Z_t}  \\
&= \int_0^{\cdot} \pare{a-\theta'm +\theta'\dvg{c}-\nabla\cdot(c\theta) + \frac{1}{2}(\theta'c\theta + \eta'\eta) }(\ze^T_t)
\ud t - \int_0^{\cdot} \pare{\eta(\ze^T_t)' \ud B^T_t +
  \theta'(\ze^T_t) \ud \ze^T_t},\\
&= \int_0^{\cdot} \pare{a-\theta'c\frac{\nabla p}{p} - \nabla\cdot (c\theta) +
  \frac{1}{2}(\theta'c\theta + \eta'\eta) }(\ze^T_t)
\ud t - \int_0^{\cdot} \pare{\eta(\ze^T_t)'\ud B^T_t +
  \theta'\sigma(\ze^T_t) \ud W^T_t}.
\end{align*}
The fact that $D = \exp(- R)$ gives $\Delta^T = \exp(- \rho^T)$. Then, the dynamics for $\Delta^T$ follow  from the dynamics of $\rho^T$.
\end{proof}

\begin{prop}\label{prop: zeta_chi_t_dynamics}
Let Assumptions \ref{ass: standing} hold.  Then, for each $T>0$ there is a filtration $\bG^T$ satisfying the usual conditions and $d$ and $k$ dimensional independent $(\prob,\bG^T)$ Brownian motions $W^T,B^T$ on $[0,T]$ so that the pair $(\ze^T,\chi^T)$ have dynamics
\begin{equation}\label{eq: ze_chi_dyn}
\begin{split}
\ze^T_t & = \ze^T_0 + \int_0^T\left(c\frac{\nabla p}{p} + \dvg{c} - m\right)(\ze^T_u)\ud u + \int_0^T\sigma(\ze^T_u)\ud W^T_u,\\
\chi^T_t& = \chi^T_0 + \int_0^T \pare{f(\ze^T_u)
  -\chi^T_u\pare{a-\theta'c\frac{\nabla p}{p} - \nabla\cdot(c\theta)}(\ze^T_u)}\ud u\\
&\qquad\qquad  + \int_0^T \chi^T_u\pare{\theta'\sigma(\ze^T_u)\ud W^T_u+ \eta(\ze^T_u)'\ud B^T_u}.
\end{split}
\end{equation}
\end{prop}

\begin{proof}[Proof of Proposition \ref{prop: zeta_chi_t_dynamics}]
Proposition \ref{prop: zeta_t_dynamics} immediately implies that under
$\prob$, $\ze^T$ has dynamics:
\begin{equation}\label{eq: zeta_dynamics}
\begin{split}
\ze^T_t &=\ze^T_0 + \int_0^t \left(c\frac{\nabla p}{p} +\dvg{c} - m\right)(\ze^T_u)du + \int_0^t
\sigma(\ze^T_u)dW^T_u;\qquad t\in [0,T],
\end{split}
\end{equation}
where $(W^T_t)_{t \in [0, T]}$ is a Brownian motion on $(\Omega, \bG^T,
\prob)$. In order to
specify the dynamics for $\chi^T$, recall the definition of $\Delta^T$ from
\eqref{eq: Delta_def}. Observe that
\[
X_{T - t} = \frac{1}{D_{T - t}} \int_{T - t}^\infty D_u f(Z_u) \ud u =
\frac{D_T}{D_{T - t}} \pare{X_T + \int_{T - t}^T \frac{D_u}{D_T} f(Z_u) \ud
  u};\qquad t \in [0, T].
\]
Then, using the definitions of $\chi^T$, $\ze^T$ and $\Delta^T$, the above is rewritten as
\begin{equation} \label{eq: chiT_closed_form}
\chi^T_{t} = \Delta_t^T \pare{\chi^T_0 + \int_0^t \frac{1}{\Delta^T_u} f(\ze^T_u) \ud u}; \qquad t \in [0, T].
\end{equation}
Lemma \ref{lem: DeltaT_dyn} implies $\Delta^T$ is a
semimartingale, and hence \eqref{eq: chiT_closed_form} yields
\[
\chi^T_{t} = \chi^T_0 + \int_0^t \chi^T_{u} \, \frac{\ud \Delta_u^T} {\Delta_u^T} + \int_0^t f(\ze^T_u) \ud u; \qquad t \in [0, T].
\]
The result now follows by plugging in for $\ud \Delta^T_u / \Delta^T_u$ from
\eqref{eq: Delta_dyn}.
\end{proof}

\section{Proof of Theorem \ref{thm:main}}\label{sec: ze_chi_ergodic}

\subsection{Preliminaries}\label{subsec: ergodic_prelim} We first prove two technical results. The first asserts the existence of a probability space and stationary processes $(\ze,\chi)$ consistent with $(\ze,\chi^x)$ in Theorem \ref{thm:main} in that given $\chi_0 = x$, it holds that $\chi_t = \chi^{x}_t, t\geq 0$. The second proposition shows that under the non-degeneracy assumption $|\eta|(z) > 0, z\in E$ and regularity assumption $f\in C^2(E;\Real_+)$ it follows that $(\ze,\chi)$ is ergodic.

\begin{lem}\label{lem: horizon_T_not_nec}
Let Assumption \ref{ass: standing} hold. Then, there is a filtered probability space $\left(\Omega,\bF,\qprob\right)$, supporting independent $d$ and $k$ dimensional Brownian motions $W$ and $B$, $\F_0$ measurable random variables $\ze_0,\chi_0$ with joint distribution $\pi$, as well as a stationary process $\ze$ with dynamics
\begin{equation}\label{eq: ze_dyn_noT}
\ze = \ze_0 + \int_0^\cdot \left(c\frac{\nabla p}{p} + \dvg{c} - m\right)(\ze_t)\ud t + \int_0^\cdot \sigma(\ze_t)\ud W_t.
\end{equation}
Furthermore, with $\Delta,\chi^x$ defined as in \eqref{eq: Delta_no_T_def}, \eqref{eq: chi_no_T_def}, if the process $\chi$ is defined by $\chi_t \dfn \chi^{\chi_0}_t$ (see Remark \ref{rem: chi_via_zeta}) then $(\ze,\chi)$ are stationary with invariant measure $\pi$ and joint dynamics
\begin{equation}\label{eq: ze_chi_dyn_noT}
\begin{split}
\ud \ze_t & = \left(c\frac{\nabla p}{p} + \dvg{c} - m \right)(\ze_t)\ud t + \sigma(\ze_t)\ud W_t, \quad \tir, \\
\ud \chi_t& = \pare{f(\ze_t) -\chi_t\pare{a-\theta'c\frac{\nabla p}{p} - \nabla\cdot(c\theta)}(\ze_t)}\ud t + \chi_t\left(\theta'\sigma(\ze_t)\ud W_t + \eta(\ze_t)'\ud B_t\right), \quad \tir.
\end{split}
\end{equation}
\end{lem}

\begin{proof}
This result follows from Proposition \ref{prop: zeta_t_dynamics}. Indeed, one can start with a probability space $\left(\Omega,\bF,\qprob\right)$ supporting independent $d$ and $k$ dimensional Brownian motions $W$ and $B$ respectively, as well as a $\F_0$ measurable random variable $(\ze_0,\chi_0)\sim \pi$ (hence independent of $W$ and $B$).  Under the given regularity assumptions, Proposition \ref{prop: zeta_t_dynamics} yields a strong, stationary solution $\ze$ satisfying \eqref{eq: ze_dyn_noT}. Then, defining $\Delta$ as in \eqref{eq: Delta_def} and, for $x>0$, $\chi^x$ as in \eqref{eq: chi_no_T_def}, it follows that $(\ze,\chi^x)$ and hence $(\ze,\chi)$ satisfy the SDE in \eqref{eq: ze_chi_dyn_noT}. Under the given regularity assumptions the law under $\prob$ of $(\ze^T,\chi^T)$ given $\ze^T_0 = z, \chi^T_0 = x$ coincides with the law under $\qprob$ of $(\ze,\chi^x)$ given that $\ze_0 = z$.  Since by construction, $\pi$ is an invariant measure for $(\ze^T,\chi^T)$, it follows from the Markov property that $\pi$ is invariant for $(\ze,\chi)$ under $\qprob$ and hence $(\ze,\chi)$ is stationary with invariant measure $\pi$.
\end{proof}

Define the measures $\qprob^{z,x}$ for $(z,x)\in F$ via
\begin{equation}\label{eq: probtx_defn}
\qprob^{z,x}\bra{A} = \qprob\bra{A\such \ze_0 = z,\chi_0 = x};\qquad A\in\F_{\infty}
\end{equation}

We now consider when $|\eta| > 0$ on $E$ and $f\in C^2(E;\Real_+)$. According to Theorem \ref{thm: pde_soln}, $g\in C^{2,\gamma}(F)$ and hence $\pi$ possesses a density satisfying
\begin{equation}\label{eq: pi_density_val}
\pi(z,x) = p(z)\partial_x g(z,x);\qquad (z,x)\in F.
\end{equation}
Additionally, we have the following Proposition:

\begin{prop}\label{prop: backwards_ergodicity}
Let Assumption \ref{ass: standing} hold, and additionally suppose that $|\eta|(z) > 0$ for $z\in E$ and $f\in C^2(E;\Real_+)$. Then the process $(\ze,\chi)$ from Lemma \ref{lem: horizon_T_not_nec} is ergodic. Thus, for all bounded measurable functions $h$ on $F$ and all $(z,x)\in F$
\begin{equation}\label{eq: p-ergodic_AA}
\lim_{T\to\infty}\frac{1}{T}\int_0^T h(\ze_t,\chi_t)dt = \int_F hd\pi;\quad \qprob^{z,x} \textrm{a.s.}.
\end{equation}
%In particular, for each $(z,x)\in F$ there exists a $\Omega^{z,x}\in\F_{\infty}$ such that $\qprob^{z,x}\bra{\Omega^{z,x}} = 1$ and for $\omega\in\Omega^{z,x}$ it holds that $\hat{\pi}^{x}_T(\tilde{\omega}) \rightarrow \pi$ weakly: i.e. for all continuous bounded functions $h$ on $F$, $\lim_{T\uparrow\infty} \langle h,L_T\rangle (\tilde{\omega}) = \langle h,\pi\rangle$.
\end{prop}

%----begin commenting out -----%

%\begin{rem}\label{rem: pasting}

%The purpose of constructing $\widetilde{\Omega}^{z,x}$ is to find \emph{a single set} of probability one where the convergence in \eqref{eq: p-ergodic_AA} holds for a wide class of functions, since a prori the set of $\probt^{z,x}$ probability one in \eqref{eq: p-ergodic_AA} may depend upon the function $h$.  In general, it is impossible to find a single set which works for all bounded measurable functions. However, it is possible to find a set for all continuous bounded functions: this follows by the tightness of $\pi$ on $F$ and the separability of $C_b(D;\Real)$ (the space of continuous bounded functions with the supremum norm) for all compact $D\subset F$ - see \cite{MR2760872}.

%\end{rem}

%----end commenting out ----%

\begin{proof}[Proof of Proposition \ref{prop: backwards_ergodicity}]

Recall $A$ from \eqref{eq: bA_def} and define $b^R: F\mapsto\Real^{d+1}$ by
\begin{equation}\label{eq: b_R_def}
\begin{split}
b^R(z,x) \dfn \left(\begin{array}{c} \pare{c (\nabla p / p) + \dvg{c} - m} (z) \\ f(z) -
    x\left(a - \theta'c\pare{\nabla p/p} - \nabla\cdot(c\theta)\right)(z)\end{array}\right).
\end{split}
\end{equation}
From \eqref{eq: ze_chi_dyn_noT} it is clear that the generator for $(\ze,\chi)$ is $L^R \dfn (1/2)A^{ij}\partial^2_{ij} + (b^R)^i \partial_i$. As an abuse of notation, let $(\qprob^{z,x})_{(z,x)\in F}$ also denote the solution to the generalized martingale problem for $L^R$ on $F$. Using Theorem \ref{thm: pde_soln}, and the fact that under the given coefficient regularity assumptions, $g\in C^3(F)$ (see \cite[Ch. 6]{MR1814364}) a lengthy calculation performed in Lemma \ref{lem: tilde_LR_calc} below shows that the density $\pi$ from \eqref{eq: pi_density_val} solves $\tilde{L}^R \pi = 0$ where $\tilde{L}^R$ is the formal adjoint to $L$.  Since by construction, $\iint_{F}\pi(z,x)dz dx = 1$, positive recurrence will follow once it is shown that $(\qprob^{z,x})_{(z,x)\in F}$ is recurrent. By Proposition \ref{prop: zeta_t_dynamics}, the restriction of $\qprob^{z,x}$ to the first $d$ coordinates (i.e. the part for $\ze$) is positive recurrent. Since by \eqref{eq: chi_no_T_def} it is evident that $\chi$ does not hit $0$ in finite time, it follows that
that $\chi$ does not explode under $\qprob^{z,x}$. Thus, \cite[Corollary 4.9.4]{MR1326606} shows that $(\ze,\chi)$ is recurrent. Now, that \eqref{eq: p-ergodic_AA} holds follows from \cite[Theorem 4.9.5]{MR1326606}.

\end{proof}

\subsection{Proof of Theorem \ref{thm:main}}

The proof of Theorem \ref{thm:main} uses a number of approximations arguments. To make these arguments precise, we first enlarge the original probability space $\basisp$ so that it contains a one dimensional Brownian motion $\hat{B}$ which is independent of $Z_0,W$ and $B$. Let $D$ be as in \eqref{eq: D_defn}, and for $\eps > 0$, define $D^{\eps}\dfn D \mathcal{E}(\sqrt{\eps}\hat{B})$. Similarly to \eqref{eq: X_def} define
\begin{equation}\label{eq: X_eps_def}
X_0^\eps \dfn \int_0^\infty D^{\eps}_t f(Z_t)dt.
\end{equation}
Note that $D^{\eps}$ takes the form \eqref{eq: D_defn} for $\eta^{\eps}(z) = (\eta(z),\sqrt{\eps})$ and when the Brownian motion $B$ therein is the $k+1$ dimensional Brownian motion $(B,\hat{B})$.  Note that  $|\eta^{\eps}|^2 = |\eta|^2 + \eps > 0$.  Denote by $\pi^{\eps}$ the joint distribution of $(Z_0,X_0^{\eps})$ under $\prob$ and by $g^\eps$ the conditional cdf of $X_0^{\eps}$ given $Z_0 = z$.  By Theorem \ref{thm: pde_soln} it follows that $g^{\eps}\in C^{2,\gamma}(F)$ and hence $\pi^{\eps}$ admits a density.

In a similar manner, by enlarging the probability space $(\Omega,\bF,\qprob)$ of Lemma \ref{lem: horizon_T_not_nec} to include a Brownian motion (still labeled $\hat{B}$) which is independent of $\ze_0$, $\chi_0$, $W$ and $B$ and defining the family of processes $(\Delta^{\eps})_{\eps > 0}$ and $(\chi^{\eps,x})_{\eps > 0}$ for $x>0$ according to
\begin{equation}\label{eq: delta_chi_new_eps}
\begin{split}
\Delta^{\eps}_t &\dfn \Delta_t\mathcal{E}(\sqrt{\eps}\hat{B})_t;\qquad t\geq 0\\
\chi^{\eps,x}_{t} &\dfn \Delta^\eps_t \pare{x + \int_0^t \frac{1}{\Delta^\eps_u} f(\ze_u) \ud u}; \qquad t\geq 0,
\end{split}
\end{equation}
it follows that $(\ze,\chi^{x,\eps})$ solve the SDE
\begin{equation}\label{eq: ze_chi_dyn_noT_eps}
\begin{split}
d\ze_t & = \left(m+2\xi\right)(\ze_t)\ud t + \sigma(\ze_t)\ud W_t,\\
d\chi^{\eps,x}_t& = \pare{f(\ze_t) -\chi^\eps_t\pare{a-\theta'c\frac{\nabla p}{p} - \nabla\cdot (c\theta)}(\ze_t)}\ud t + \chi^\eps_t\pare{\theta'\sigma(\ze_t)\ud W_u + \eta^{\eps}(\ze_t)'(\ud B_t,\ud \hat{B}_t)}.
\end{split}
\end{equation}
Since $|\eta^\eps|\geq \sqrt{\eps} >0$, Proposition \ref{prop: backwards_ergodicity} shows for $f\in C^2(E;\Real_+)$ the generator $L^{\eps,R}$ associated to \eqref{eq: ze_chi_dyn_noT_eps} is positive recurrent with invariant density $\pi^{\eps}$ and thus for all $(z,x)\in F$ and all bounded measurable functions $h$ on $F$ (note that conditioned upon $\chi_0 = x$ we have $\chi^{\eps,x}_0=\chi^{x}_0 = x = \chi_0$):
\begin{equation}\label{eq: p-ergodic_AA_eps}
\lim_{T\to\infty}\frac{1}{T}\int_0^T h(\ze_t,\chi^{x,\eps}_t)dt = \int_F hd\pi^{\eps};\qquad \qprob^{z,x} a.s..
\end{equation}

%Lastly, since the proof of Theorem \ref{thm: conv_for_chi} below uses an approximation argument, we note that by again enlarging the space $\basisttzx$ to include an additional Brownian motion $\hat{B}$ which is independent of both $W$ and $B$ and, for $\eps > 0$ and $\Delta$ as in \eqref{eq: Delta_no_T_def}, defining

%Thus, in the setting of \eqref{eq: Z_defn}, \eqref{eq: D_defn} we see that $(\ze_t,\chi^\eps_t)$ corresponds to the reversed dynamics when $\eta$ from \eqref{eq: D_defn} is equal to $(\eta(z),\sqrt{\eps})$ and the Brownian motion $B$ within \eqref{eq: D_defn} is  a $k+1$ dimensional Brownian motion independent of $Z_0$ and $W$.  Note that by construction $\eta'\eta(z) + \eps >0$ on $E$ and hence in this setting Assumption \ref{ass: non-degeneracy} holds. Lastly, note that one can also regularize $\eta$ in the original (potentially enlarged) probability space.  Here, take $\hat{B}$ to be a Brownian motion independent of $Z_0,W$ and $B$.  Then, for a given $\eps > 0$, define $D^{\eps}_t \dfn D_t \mathcal{E}\pare{\sqrt{\eps}\hat{B}}_t$ for $t\geq 0$ and $X^{\eps}$ according to \eqref{eq: X_def} with $D^{\eps}$ in place of $D$.  Then $D^{\eps}$ satisfies \eqref{eq: D_defn} with $\eta$ therein equal to $(\eta,\sqrt{\eps})$ and $B$ a $k+1$ dimensional Brownian motion independent of $Z_0$ and $W$.  Thus, this $X^{\eps}$ corresponds to the locally elliptic case and it is clear that the reversed process $(\ze,\chi^{\eps})$ has generator consistent with the dynamics in \eqref{eq: ze_chi_dyn_noT_eps}.

With all the notation in place, Theorem \ref{thm:main} is the culmination of a number of lemmas, which are now presented. The first lemma implies that $\pi^{\eps}$ converges weakly to $\pi$ as $\eps\downarrow 0$.

\begin{lem}\label{lem: pi_eps_to_pi}

Let Assumption \ref{ass: standing} hold. Define $X_0^{\eps}$ as in \eqref{eq: X_eps_def}. Then $X_0^\eps$ converges to $X$ in $\prob$-measure as $\eps\to 0$.

\end{lem}

\begin{proof}[Proof of Lemma \ref{lem: pi_eps_to_pi}]

Denote by $\G$ the sigma-field generated by $Z_0$, $W$ and $B$. Set $\delta^{\eps}_t \dfn  D^\eps_t/D_t = \mathcal{E}\left(\sqrt{\eps}\hat{B}_t\right)$. By the independence of $\delta^\eps$ and $\G$:
\[
\expec \bra{|X_0^{\epsilon} - X_0| \such \G} \leq \int_0^\infty \expec \bra{|\delta^{\epsilon}_t - 1| \such \G}D_t f(Z_t) \ud t = \int_0^\infty \expec \bra{|\delta^{\epsilon}_t - 1|} D_t f(Z_t) \ud t.
\]
Now, set $h^\eps_t\dfn \sqrt{e^{\eps t} -1}$. Note that $h^{\eps}$ is monotone increasing in $\eps$ with $\lim_{\eps \to 0} h^{\eps} = 0$. Furthermore,
\[
\expec \bra{|\delta^{\eps}_t - 1|} \leq \expec \bra{|\delta^{\eps}_t - 1|^2}^{1/2} = \sqrt{\exp(\eps t) - 1} = h^{\eps}_t.
\]
By assumption, $\prob\bra{X_0<\infty} = 1$.  Since for any $\eps > 0$, $\sup_{t\geq 0} \delta^\eps_t < \infty$ $\prob$ a.s., it thus follows that $\prob\bra{X_0^\eps < \infty} = 1$. The dominated convergence theorem applied path-wise (recall that there exists a $\kappa > 0$ so that $e^{\kappa t} D_t \rightarrow 0$ $\prob$ almost surely) then gives that $\lim_{\eps \to 0} \expec \bra{|X_0^{\eps} - X_0| \such \G} = 0$, which shows that the pair $(Z_0, X_0^{\eps})$ converges in probability to $(Z_0, X_0)$, finishing the proof.
\end{proof}

Next, define $\mathcal{C}$ as the class of (Borel measurable) functions $h$ which are bounded and Lipschitz in $x$, uniformly in $z$; in other words,
\begin{equation}\label{eq: lip-est_A}
\mathcal{C} \dfn \cbra{h\in \B(E;\Real) \such \exists K(h)>0 \ s.t.\ \forall x_1,x_2 > 0,\ \sup_{z\in E} |h(z,x_1) - h(z,x_2)| \leq K(h) \left(1\wedge |x_1 - x_2|\right)}.
\end{equation}
The next Lemma gives a weak form of the convergence in Theorem \ref{thm:main} for regular $f$. Note that the notation $\qprob\text{-}\lim_{T\to\infty}$ stands for the limit in $\qprob$ probability as $T\to\infty$.

\begin{lem}\label{lem: deg_conv_in_prob}

Let Assumption \ref{ass: standing} hold. Assume additionally that $f\in C^2(E;\Real_+)$. Then for all $x > 0$ and all $h\in\mathcal{C}$:
\begin{equation}\label{eq: p-ergodic}
\qprob\text{-}\lim_{T\to\infty} \frac{1}{T}\int_0^T h(\ze_t,\chi^x_t)dt  = \int_{F}hd\pi.
\end{equation}

\end{lem}

\begin{proof}[Proof of Lemma \ref{lem: deg_conv_in_prob}]
For ease of presentation we adopt the following notational conventions. First, for any measurable function $f$ and probability measure $\nu$ on $F$ set
\begin{equation}\label{eq: pi_int}
\langle h,\nu\rangle \dfn \int_F hd\nu.
\end{equation}
Next, similarly to $\hat{\pi}^x_T$ in \eqref{eq: LT_def}, we define $\hat{\pi}^{\eps,x}_T$ to be the empirical measure of $(\ze,\chi^{\eps,x})$ on $[0,T]$ for $\chi^{\eps,x}$ as in \eqref{eq: delta_chi_new_eps}. Thus, we write
\begin{equation*}
\frac{1}{T}\int_0^T h(\ze_t,\chi^x_t)dt = \langle h,\hat{\pi}^x_T\rangle;\qquad \frac{1}{T}\int_0^T h(\ze_t,\chi^{\eps,x}_t)dt = \langle h,\hat{\pi}^{\eps,x}_T\rangle.
\end{equation*}
Proposition \ref{prop: backwards_ergodicity} implies for all $x>0$ and $\eps > 0$ that
\[
\qprob\text{-}\lim_{T\to\infty} \inner{h}{\hat{\pi}^{\eps,x}_T} = \inner{h}{\pi^\eps}.
\]
Indeed, \eqref{eq: p-ergodic_AA_eps} gives for all $(z,x)\in F$:
\begin{equation}
\lim_{T\to\infty}\langle h, \hat{\pi}^{\eps,x}_T\rangle = \inner{h}{\pi^{\eps}};\qquad \qprob^{z,x} \ a.s..
\end{equation}
Thus, the above limit holds $\qprob$ almost surely, and hence in probability.

%the given regularity assumptions imply that the law of $(\ze,\chi^\eps)$ under $\probt$ given that $\ze_0 = z, \chi^{\eps}_0 = x$, coincides with $\qprob^{z,x}_R(\eps)$: the solution to the martingale problem for $L^{R}(\eps)$ on $F$ where $L^{R}(\eps)$ is the generator associated to the SDE in \eqref{eq: ze_chi_dyn_noT_eps}.  According to Theorem \ref{thm: backwards_ergodicity} (see the proof therein) $\qprob_R^{z,x}(\eps)$ is positive recurrent with invariant measure $\pi^{\eps}$. Thus, for any $\delta > 0$
%\begin{equation*}
%\probt\bra{\left|\inner{h}{L^\eps_T} - \inner{h}{\pi^{\eps}}\right|\geq \delta\such \ze_0 = z, \chi^{\eps}_0 = x} = %\qprob^{z,x}_R(\eps)\bra{\left|\frac{1}{T}\int_0^T h(\ze_t,\chi^{\eps}_t)dt - \inner{h}{\pi}\right|\geq \delta},
%\end{equation*}
%where on the right hand side of the above, $(\ze,\chi^{\eps})$ is the coordinate mapping process on the path-space $(C[0,\infty);F)$. The ergodic theorem applied to $\qprob^{z,x}_R(\eps)$ implies the latter
%probability tends to $0$ as $T\to\infty$.  Since this holds for all $z \in E$ it also holds when $\ze_0\sim p$ by the bounded convergence theorem and hence it holds under $\probt^{x}$.

To prove \eqref{eq: p-ergodic} we need to show that for any increasing $\Real_+$-valued sequence $(T_n)_{\nin}$ such that $\limn T_n = \infty$, there is a sub-sequence $(T_{n_k})_{\kin}$ such that
\[
\qprob\text{-} \lim_{k\to\infty} \inner{h}{\hat{\pi}^x_{T_{n_k}}} = \inner{h}{\pi},
\]
as this implies \eqref{eq: p-ergodic} by considering double sub-sequences. To this end, let $(\eps_k)_{k \in \Natural}$ be any strictly positive sequence that converges to zero, and assume that $\eps_1 < \kappa$, where $\kappa > 0$ is from Assumption (A5). Next, pick $T_{n_k}$ large enough so that $k/T_{n_k} \rightarrow 0$ and such that
\[
\qprob \bra{\abs{\inner{h}{\hat{\pi}^{\eps_k,x}_{T_{n_k}}} -\inner{h}{\pi^{\eps_k}} } > \frac{1}{k}} \leq \frac{1}{k}.
\]
As argued above, this is possible since $\inner{h}{\hat{\pi}^{\eps_k,x}_T}$ converges to $\inner{h}{\pi^{\eps_k}}$ in $\qprob$ probability.  Since Lemma \ref{lem: pi_eps_to_pi} implies $\lim_{\eps\to 0}\inner{h}{\pi^{\eps_k}}= \inner{h}{\pi}$ it follows that
\[
\qprob\text{-}\lim_{k\to\infty} \inner{h}{\hat{\pi}^{\eps_k,x}_{T_{n_k}}} = \inner{h}{\pi}.
\]
Since
\begin{equation*}
\begin{split}
\left|\inner{h}{\hat{\pi}^x_{T_{n_k}}} - \inner{h}{\pi}\right| &\leq \left| \inner{h}{\hat{\pi}^x_{T_{n_k}}} - \inner{h}{\hat{\pi}^{\eps_k,x}_{T_{n_k}}}\right| + \left|\inner{h}{\hat{\pi}^{\eps_k,x}_{T_{n_k}}} - \inner{h}{\pi}\right|,
\end{split}
\end{equation*}
it suffices to show
\[
\qprob\text{-}\lim_{k\to\infty}\left|\inner{h}{\hat{\pi}^{\eps_k,x}_{T_{n_k}}} - \inner{h}{\hat{\pi}^x_{T_{n_k}}}\right| = 0.
\]
In fact, the claim is that
\[
\lim_{k \to \infty} \expec^{\qprob} \bra{ \abs{\inner{h}{\hat{\pi}^{\eps_k,x}_{T_{n_k}}}- \inner{h}{\hat{\pi}^x_{T_{n_k}}}}  } = 0,
\]
or the even stronger (recall $\inner{h}{\hat{\pi}^x_T} = (1/T)\int_0^T h(\ze_t,\chi^x_t)\ud t$, $\inner{h}{\hat{\pi}^{\eps_k,x}_T} = (1/T)\int_0^T h(\ze_t,\chi^{\eps_k,x}_t)\ud t$):
\begin{equation} \label{eq: L^1-est}
\lim_{k \to \infty}  \pare{ \frac{1}{T_{n_k}} \int_0^{T_{n_k}} \expec^{\qprob} \bra{ \left|h(\zeta_t, \chi^{\eps_k,x}_t) - h(\zeta_t, \chi^x_t)\right| } \ud t  } = 0.
\end{equation}
From \eqref{eq: lip-est_A}:
\begin{equation}\label{eq: L^1-est_a}
\frac{1}{T_{n_k}} \int_0^{T_{n_k}} \expec^{\qprob} \bra{ \left|h(\zeta_t, \chi^{\eps_k,x}_t) - h(\zeta_t, \chi_t^x)\right| } \ud t \leq \frac{K}{T_{n_k}}\int_0^{T_{n_k}} \expec^{\qprob}\bra{1\wedge \left|\chi^{\eps_k,x}_t - \chi^x_t\right|}\ud t.
\end{equation}
Furthermore, recall that
\[
\chi^x_t = \Delta_t\left(x + \int_0^t \frac{1}{\Delta_u}f(\ze_u)\ud u\right), \quad \chi^{\eps_k,x}_t = \Delta^{\eps_k,x}_t\left(x + \int_0^t \frac{1}{\Delta^{\eps_k}_u}f(\ze_u)\ud u\right),
\]
where $\Delta^{\eps_k}$ is from \eqref{eq: delta_chi_new_eps}. With $\delta^{\eps_k}\dfn \mathcal{E}\pare{\sqrt{\eps_k}\hat{B}}$ it follows that under $\qprob$
\begin{align*}
|\chi^{\eps_k,x}_t - \chi^x_t| &\leq x |\Delta^{\eps_k}_t - \Delta_t| + \int_0^t \abs{\frac{\Delta^{\eps_k}_t}{\Delta^{\eps_k}_u} - \frac{\Delta_t}{\Delta_u}} f(\zeta_u) \ud u \\
& = x\Delta_t |\delta^{\eps_k}_t - 1| + \int_0^t \frac{\Delta_t}{\Delta_u} \abs{\frac{\delta^{\eps_k}_t}{\delta^{\eps_k}_u} - 1} f(\zeta_u) \ud u.
\end{align*}
With $\G$ now denoting the $\sigma$ field generated by $\ze_0$, $W$ and $B$, by the independence of $\hat{B}$ and $\G$ it follows that
\begin{equation}\label{eq: cond_exp_ub}
\expec^{\qprob} \bra{|\chi^{\eps_k,x}_t - \chi^{x}_t| \such \G} \leq x \Delta_t h^{\eps_k}_t + \int_0^t \frac{\Delta_t}{\Delta_u} h^{\eps_k}_{t -u} f(\zeta_u) \ud u.
\end{equation}
where for any $\eps > 0$, $h^{\eps}$  is from Lemma \ref{lem: pi_eps_to_pi}. Since $\ze$ is stationary under $\qprob$,  it holds for all $t>0$ that the distribution of $\Delta_t$ under $\qprob$ coincides with the distribution of  $D_t$ under $\prob$ and the distribution of $\int_0^t (\Delta_t / \Delta_u) h^{\eps_k}_{t -u} f(\zeta_u) \ud u$ under $\qprob$ is the same as the distribution of $\int_0^t D_u h^{\eps_k}_{u} f(Z_u) \ud u$ under $\prob$.

We next claim there exists a sequence $\delta_k \to 0$ such that
\begin{equation}\label{eq: good_delta_k}
\sup_{t \in [k, \infty)} \prob \bra{1\wedge\left(x D_t h^{\eps_k}_t + \int_0^t D_u h^{\eps_k}_{u} f(Z_u) \ud u\right)   > \delta_k } \leq \delta_k, \quad \forall k \in \Natural.
\end{equation}
This is shown at the end of the proof.  Admitting this fact, and using $\expec^{\qprob} \bra{1 \wedge |\chi^{\eps_k,x}_t - \chi^{x}_t| \such \G} \leq 1 \wedge \expec^{\qprob} \bra{|\chi^{x,\eps_k}_t - \chi^{x}_t| \such \G}$, it follows that
\begin{equation*}
\begin{split}
&\lim_{k \to \infty} \pare{ \sup_{t \in [k, \infty)} \expec^{\qprob} \bra{ 1 \wedge |\chi^{\eps_k,x}_t - \chi^{x}_t|  }} =  \lim_{k \to \infty} \pare{ \sup_{t \in [k, \infty)} \expec^{\qprob} \bra{\expec^{\qprob} \bra{1 \wedge |\chi^{\eps_k,x}_t - \chi^{x}_t| \such \G}  } }\\
&\leq\lim_{k\to\infty} \pare{\sup_{t\in[k,\infty)} \expec\bra{1\wedge\left(x D_t h^{\eps_k}_t + \int_0^t D_u h^{\eps_k}_{u} f(Z_u) \ud u\right)}}\\
&\leq \lim_{k\to\infty}2\delta_k = 0.
\end{split}
\end{equation*}
In the above, the first inequality holds because of \eqref{eq: cond_exp_ub} and the second by \eqref{eq: good_delta_k} and the fact that for any r.v. $Y$, $\expec\bra{1\wedge Y} \leq \delta + \prob\bra{1\wedge Y > \delta}$.  The last equality follows by construction of $\delta_k$. Recall that $T_{n_k}$ was chosen so that $\lim_{k \to \infty} (k / T_{n_k}) = 0$ , it follows that
\begin{align*}
\limsup_{k \to \infty} \pare{ \frac{1}{T_{n_k}} \int_0^{T_{n_k}} \expec^{\qprob} \bra{ 1 \wedge | \chi^{\epsilon_k,x}_t - \chi^x_t| } \ud t } &\leq \limsup_{k\to\infty}\left(\frac{k}{T_{n_k}} + \frac{T_{n_k}-k}{T_{n_k}}\sup_{t \in [k, \infty)} \expec^{\qprob} \bra{ 1 \wedge |\chi^{\eps_k,x}_t - \chi^{x}_t|  }\right) \\
&= 0.
\end{align*}
which in view of \eqref{eq: L^1-est_a} implies \eqref{eq: L^1-est}, finishing the proof. Thus, it remains to show \eqref{eq: good_delta_k}.  Since for any $a,b>0$, $1\wedge(a+b)\leq 1\wedge a + 1\wedge b$ the two terms on the right hand side of \eqref{eq: good_delta_k} are treated separately.  Let $\delta_k > 0$.  First we have
\begin{align*}
\prob\bra{1\wedge x D_t h^{\eps_k}_t > \delta_k} &\leq \prob\bra{x D_t h^{\eps_k}_t > \delta_k}\\
&= \prob\bra{x D_t e^{\kappa t} > \delta_k e^{\kappa t}/h^{\eps_k}_t}
\end{align*}
Now, $h^{\eps_k}_t \leq  e^{\eps_k /2 t}$ so on $t\geq k$, $e^{\kappa t}/h^{\eps_k}_t \geq e^{(\kappa-\eps_k/2)t} \geq e^{(\kappa-\eps_k/2)k}$ since $\eps_k /2 < \kappa$.  So, for any $\delta_k > e^{-(\kappa-\eps_k/2)(k/2)}$ it follows that
\begin{equation*}
\prob\bra{x D_t h^{\eps_k}_t > \delta_k} \leq \prob\bra{ x D_t e^{\kappa t} \geq e^{(\kappa-\eps_k/2)(k/2)}}
\end{equation*}
Set $\tilde{\delta}_k \dfn \sup_{t\geq k} \prob\bra{x D_t e^{\kappa t} \geq e^{(\kappa-\eps_k/2)(k/2)}}$.  Since $D_t e^{\kappa t}$ goes to $0$ in $\prob$ probability, it follows that $\tilde{\delta}_k \rightarrow 0$.  Thus, taking $\delta_k$ to be maximum of $\tilde{\delta}_k$ and $e^{-(\kappa-\eps_k/2)(k/2)}$ it follows that
\begin{equation*}
\prob\bra{1\wedge \chi D_t h^{\eps_k}_t > \delta_k} \leq \delta_k.
\end{equation*}
Turning to the second term in \eqref{eq: good_delta_k}, it is clear that
\begin{equation*}
1\wedge \int_0^t D_u h^{\eps_k}_u f(Z_u)\ud u \leq 1 \wedge \int_0^\infty D_u h^{\eps_k}_u f(Z_u)\ud u
\end{equation*}
As shown in the proof of Lemma \ref{lem: pi_eps_to_pi}, $\int_0^\infty D_u h^{\eps_k}_uf(Z_u)\ud u$ goes to $0$ as $k\to\infty$ almost surely. Thus by the bounded convergence theorem, $\expec\bra{1\wedge \int_0^\infty D_u h^{\eps_k}_u f(Z_u)\ud u} \rightarrow 0$ as $k\to\infty$.  Since
\begin{equation*}
\prob\bra{1\wedge \int_0^\infty D_u h^{\eps_k}_u f(Z_u)\ud u > \delta_k} \leq \frac{1}{\delta_k}\expec\bra{1\wedge \int_0^\infty D_u h^{\eps_k}_u f(Z_u)\ud u},
\end{equation*}
upon defining $\delta_k \dfn \sqrt{\expec\bra{1\wedge\int_0^\infty D_u h^{\eps_k}_u f(Z_u)\ud u}}$ it follows that $\prob\bra{1\wedge \int_0^\infty D_u h^{\eps_k}_u f(Z_u)\ud u > \delta_k} \leq \delta_k$ and $\delta_k \rightarrow 0$.  This concludes the proof since to combine the two terms one can take $\delta_k$ to be twice the maximum of the $\delta_k$'s for individual terms.

\end{proof}

The next lemma proves the convergence in Lemma \ref{lem: deg_conv_in_prob} for $f\in \Lb^1(E,p)$, not just $f\in C^2(E;\Real_+)$.

\begin{lem}\label{lem: deg_conv_in_prob_gen}

Let Assumption \ref{ass: standing} hold. Then for all $x > 0$ and all $h\in\mathcal{C}$:
\begin{equation}\label{eq: p-ergodic_gen}
\qprob\text{-}\lim_{T\to\infty} \frac{1}{T}\int_0^T h(\ze_t,\chi^x_t)dt  = \int_{F}hd\pi.
\end{equation}

\end{lem}

\begin{proof}[Proof of Lemma \ref{lem: deg_conv_in_prob_gen}]

By mollifying $f$, since $p$ is tight in $E$ there exists a sequence of functions $f^n\in C^2(E)\cap \Lb^1(E,p)$ with $f^n\geq 0$ such that
\begin{equation}\label{eq: mollify}
\int_E \left|f^n(z)-f(z)\right|p(z)dz \leq n^{-2}2^{-n}.
\end{equation}
Note that
\begin{equation*}
\begin{split}
\expec\bra{\int_0^\infty ne^{-t/n}|f^n(Z_t)-f(Z_t)|\ud t} &= \int_0^\infty ne^{-t/n}\expec\bra{|f^n(Z_t)-f(Z_t)|}\ud t\\
&= \int_0^\infty ne^{-t/n}\left(\int_E |f^n(z)-f(z)|p(z)dz\right)\ud t\\
&\leq \int_0^\infty n^{-1}e^{-t/n}2^{-n}\ud t\\
&= 2^{-n}.
\end{split}
\end{equation*}
Thus, by the Borel-Cantelli lemma it follows that $\prob$ almost surely
\begin{equation*}
\lim_{n\rightarrow\infty} \int_0^\infty ne^{-t/n}|f^n(Z_t)-f(Z_t)|\ud t = 0.
\end{equation*}
For $n>\kappa$ from Assumption \ref{ass: standing}, let $A_n = n^{-1}\sup_{\tir}(e^{t/n}D_t)$. Note that $\lim_{n\rightarrow \infty} A_n = 0$ almost surely since for each $\delta>0$ we can find a $\prob$ almost surely finite random variable $T = T(\delta)$ so that $D_t \leq \delta e^{-\kappa t}$ for $t\geq T$, and hence
\begin{equation*}
A_n = \frac{1}{n}\sup_{t\in\tir}\left(e^{t/n}D_t\right) \leq \frac{1}{n}e^{T/n}\sup_{t\leq T} D_t + \frac{\delta}{n}.
\end{equation*}
Since
\begin{equation*}
\int_0^\infty D_t|f^n(Z_t)-f(Z_t)|\ud t \leq A_n\int_0^\infty ne^{-t/n}|f^n(Z_t)-f(Z_t)|\ud t
\end{equation*}
we see that
\begin{equation}\label{eq: approx_val}
\lim_{n\rightarrow\infty} \int_0^\infty D_t\left|f^n(Z_t)-f(Z_t)\right|\ud t = 0;\qquad \prob-\textrm{ a.s.}
\end{equation}
Thus, with $X^n_0\dfn \int_0^\infty D_t f^n(Z_t)\ud t$ that $\lim_{n\rightarrow\infty} X^n_0 = X_0$ almost surely and hence if $\pi^n$ is the joint distribution of $(Z_0,X^n_0)$ then $\pi^n$ converges to $\pi$ weakly, as $n\rightarrow \infty$. Now, on the same probability space as in Lemma \ref{lem: horizon_T_not_nec} define
\begin{equation*}
\chi^{x,n}_t\dfn \Delta_t\left(x + \int_0^t \Delta_t^{-1}f^n(\ze_t)\ud t\right);\quad t\geq 0.
\end{equation*}
Note that
\begin{equation*}
\left|\chi^{n,x}_t-\chi^x_t\right| \leq \Delta_t\int_0^t \Delta_u^{-1}\left|f^n(\ze_u)-f(\ze_u)\right|\ud u, \qquad \forall\ t\geq 0,
\end{equation*}
and by construction the law of the process on the right hand side above under $\qprob$ is the same as the law of $\int_0^\cdot D_u \left|f^n(Z_u)-f(Z_u)\right|\ud u$ under $\prob$.  It thus follows that for $\delta > 0$
\begin{equation*}
\sup_{\tir}\qprob\bra{\left|\chi^{n,x}_t-\chi^x_t\right| > \delta} \leq \prob\bra{\int_0^\infty D_u\left|f^n(Z_u)-f(Z_u)\right|\ud u > \delta} \dfn \phi^n(\delta).
\end{equation*}
By \eqref{eq: approx_val} we can find a non-negative sequence $(\delta_n)$ such that $\delta_n\rightarrow 0$ and $\lim_{\delta\rightarrow 0}\phi^n(\delta_n) = 0$. Now, for $h\in\mathcal{C}$ we have almost surely for $t\geq 0$:
\begin{equation*}
\left|h(\ze_t,\chi^{n,x}_t) - h(\ze_t,\chi^{x}_t)\right| \leq K \left(1\wedge \left|\chi^{n,x}_t - \chi^x_t\right|\right).
\end{equation*}
Therefore, with $\hat{\pi}^{x,n}_T$ denoting the empirical law of $(\ze,\chi^{n,x})$ we have
\begin{equation*}
\expec^{\qprob}\bra{\left|\inner{h}{\hat{\pi}^{x,n}_T} - \inner{h}{\hat{\pi}^x_T}\right|} \leq \frac{K}{T}\int_0^T\expec^{\qprob}\bra{1\wedge \left|\chi^{n,x}_t - \chi^x_t\right|}\ud t.
\end{equation*}
Since for any $0<\delta < 1$ and random variable $Y$ we have $\expec\bra{1\wedge |Y|} \leq \delta + \prob\bra{|Y|>\delta}$ it follows that for any $n$
\begin{equation*}
\sup_{T\in\Real_+}\expec^{\qprob}\bra{\left|\inner{h}{\hat{\pi}^{x,n}_T} - \inner{h}{\hat{\pi}^x_T}\right|} \leq K\left(\phi^n(\delta) + \delta\right),
\end{equation*}
and hence for the given sequence $(\delta_n)$:
\begin{equation}\label{eq: approx_val_2}
\limsup_{n\rightarrow\infty} \sup_{T\in\Real_+}\expec^{\qprob}\bra{\left|\inner{h}{\hat{\pi}^{x,n}_T} - \inner{h}{\hat{\pi}^x_T}\right|} \leq \limsup_{n\rightarrow\infty} K\left(\phi^n(\delta_n) + \delta_n\right) = 0.
\end{equation}
Now, fix an sequence $(T_k)$ such that $\lim_{k\rightarrow \infty} T_k = \infty$.  Since Lemma \ref{lem: deg_conv_in_prob} implies for each $n$, $\qprob-\lim_{T\rightarrow\infty} |\inner{h}{\hat{\pi}^{x,n}_T}-\inner{h}{\pi^n}| = 0$ for each $n$ we can find a $T_{k_n}$ so that
\begin{equation*}
\qprob\bra{\left|\inner{h}{\hat{\pi}^{x,n}_{T_{k_n}}} - \inner{h}{\pi^n}\right| > \frac{1}{n}} < \frac{1}{n}
\end{equation*}
It thus follows that
\begin{equation*}
\qprob-\lim_{n\rightarrow\infty} \left|\inner{h}{\hat{\pi}^{n,x}_{T_{k_n}}} - \inner{h}{\pi^n}\right| = 0.
\end{equation*}
Since $\lim_{n\rightarrow\infty} \left|\inner{h}{\pi^n}-\inner{h}{\pi}\right| = 0$ it follows by \eqref{eq: approx_val_2} that for each $\gamma > 0$ that
\begin{equation*}
\begin{split}
\qprob\bra{\left|\inner{h}{\hat{\pi}^x_{T_{k_n}}} - \inner{h}{\pi}\right| > \gamma} & \leq \qprob\bra{\left|\inner{h}{\hat{\pi}^x_{T_{k_n}}} - \inner{h}{\hat{\pi}^{x,n}_{T_{k_n}}}\right| > \frac{\gamma}{3}} + \qprob\bra{\left|\inner{h}{\hat{\pi}^{x,n}_{T_{k_n}}}-\inner{h}{\pi^n}\right| > \frac{\gamma}{3}}\\
&\qquad\qquad + 1_{\left|\inner{h}{\pi^n}-\inner{h}{\pi}\right| > \frac{\gamma}{3}}\\
& \leq \frac{3}{\gamma}\sup_{T\in\Real_+}\expec^{\qprob}\bra{\left|\inner{h}{\hat{\pi}^x_T} - \inner{h}{\hat{\pi}^{x,n}_T}\right|} + \qprob\bra{\left|\inner{h}{\hat{\pi}^{x,n}_{T_{k_n}}}-\inner{h}{\pi^n}\right| > \frac{\gamma}{3}}\\
&\qquad\qquad + 1_{\left|\inner{h}{\pi^n}-\inner{h}{\pi}\right| > \frac{\gamma}{3}}\\
&\rightarrow 0\textrm{ as } n\rightarrow \infty.
\end{split}
\end{equation*}
We have just showed that for any sequence $(\inner{h}{\hat{\pi}^x_{T_k}})$ there is a subsequence $(\inner{h}{\hat{\pi}^x_{T_{k_n}}})$ which converges in $\qprob$ probability to $\inner{h}{\pi}$ which in fact proves that $(\inner{h}{\hat{\pi}^x_T})$ converges in $\qprob$ probability to $\inner{h}{\pi}$, proving \eqref{eq: p-ergodic_gen}.
\end{proof}

The next lemma strengthens the convergence in Lemma \ref{lem: deg_conv_in_prob_gen} to almost sure convergence under $\qprob$, but for $\pi$ almost every $x>0$, for $h\in\mathcal{C}$ from \eqref{eq: lip-est_A}.

\begin{lem}\label{lem: ergodic_deg}

Let Assumption \ref{ass: standing} hold. Then for all $h\in \mathcal{C}$ and $\pi$ almost every $x>0$:
\begin{equation}\label{eq: p-ergodic_A}
\lim_{T\to\infty}\frac{1}{T}\int_0^T h(\ze_t,\chi^x_t)dt = \int_F h d\pi ;\qquad  \qprob \textrm{ a.s.}.
\end{equation}
\end{lem}

\begin{proof}[Proof of Lemma \ref{lem: ergodic_deg}]

We again use the notation in \eqref{eq: pi_int}. Recall $\chi$ from Lemma \ref{lem: horizon_T_not_nec} and define $\hat{\pi}_T$ as the empirical law of $(\ze,\chi)$ on $[0,T]$. Given that $(\ze,\chi)$ is stationary under $\qprob$, the ergodic theorem implies that for all bounded measurable functions $h$ on $F$ that there is a random variable $Y$ such that
\begin{equation}\label{eq: ze_chi_ergodic_g}
\lim_{T\to\infty} \inner{h}{\hat{\pi}_T} = Y;\qquad \qprob \textrm{ a.s.}.
\end{equation}
By Lemma \ref{lem: deg_conv_in_prob_gen} it holds that for $h\in \mathcal{C}$, $Y=\inner{h}{\pi}$ with $\qprob$ probability one. Indeed, let $\delta > 0$ and note:
\begin{equation*}
\begin{split}
\qprob\bra{\left|Y - \inner{h}{\pi}\right|\geq \delta} &\leq \qprob\bra{\left|Y-\inner{h}{\hat{\pi}_T}\right| + \left|\inner{h}{\hat{\pi}_T}-\inner{h}{\pi}\right| \geq \delta}\\
&\leq \qprob\bra{\left|Y-\inner{h}{\hat{\pi}_T}\right| \geq \frac{\delta}{2}} + \qprob\bra{\left|\inner{h}{\hat{\pi}_T}-\inner{h}{\pi}\right| \geq \frac{\delta}{2}}
\end{split}
\end{equation*}
The first of these two terms goes to $0$ by \eqref{eq: ze_chi_ergodic_g}. As for the second, denote by $\pi|_{x}$ the marginal of $\pi$ with respect to $\chi$. Then
\begin{equation*}
\qprob\bra{\left|\inner{h}{\hat{\pi}_T}-\inner{h}{\pi}\right| \geq \frac{\delta}{2}} = \int_0^\infty \pi|_{x}(dx)\qprob\bra{\left|\inner{h}{\hat{\pi}^x_T}-\inner{h}{\pi}\right| \geq \frac{\delta}{2}}
\end{equation*}
By Lemma \ref{lem: deg_conv_in_prob} the integrand goes to $0$ as $T\to\infty$ for all $x > 0$ and thus the result follows by the bounded convergence theorem.  Next, we have
\begin{equation*}
1 = \qprob\bra{\lim_{T\to\infty} \inner{h}{\hat{\pi}_T} = \inner{h}{\pi}} = \int_0^\infty \pi\big|_{x}(dx)\qprob\bra{\lim_{T\to\infty} \inner{h}{\hat{\pi}^x_T} = \inner{h}{\pi}},
\end{equation*}
and thus \eqref{eq: p-ergodic_A} holds for $\pi$ a.e. $x>0$, finishing the proof.
\end{proof}

The last preparatory lemma strengthens Lemma \ref{lem: ergodic_deg} to show almost sure convergence for all starting points $x>0$, not just $\pi$ almost every $x>0$.

\begin{lem}\label{lem: ergodic_deg_all_x}
Let Assumption \ref{ass: standing} hold. Then for all $h\in\mathcal{C}$ and all $x>0$
\begin{equation}\label{eq: p-ergodic_A_all_x}
\lim_{T\to\infty}\frac{1}{T}\int_0^T h(\ze_t,\chi^x_t)dt = \int_F h d\pi ;\qquad  \qprob \textrm{ a.s.}.
\end{equation}
\end{lem}

\begin{proof}[Proof of Lemma \ref{lem: ergodic_deg_all_x}]
Recall from Remark \ref{rem: chi_via_zeta} that $\chi^x$ takes the form
\begin{equation}\label{eq: chi_via_zeta_new}
\chi^x_t = \Delta_t\left(x + \int_0^t \frac{1}{\Delta_t}f(\ze_t)\ud t\right);\qquad t\geq 0.
\end{equation}
Let $h\in\mathcal{C}$. By Lemma \ref{lem: ergodic_deg}, there is some $x_0 > 0$ such that \eqref{eq: p-ergodic_A_all_x} holds.  Using the notation in \eqref{eq: pi_int} and \eqref{eq: chi_via_zeta_new} it easily follows for any $x>0$ that
\begin{equation*}
\begin{split}
\left|\inner{h}{\hat{\pi}^x_T} - \inner{h}{\hat{\pi}^{x_0}_T}\right| &\leq \frac{1}{T}\int_0^T \left|h(\ze_t,\chi^{x}_t) - h(\ze_t,\chi^{x_0}_t)\right|\ud t \leq \frac{K}{T}\int_0^T\left(1\wedge |\chi^{x}_t-\chi^{x_0}_t|\right)\ud t\\
& = \frac{K}{T}\int_0^T \left(1\wedge \Delta_t|x-x_0|\right)\ud t \leq \frac{K|x-x_0|}{T}\int_0^\infty \Delta_t \ud t
\end{split}
\end{equation*}
We will show below that $\qprob\bra{\int_0^\infty \Delta_t\ud t < \infty} = 1$.  Admitting this it holds that $\qprob$ almost surely, $\lim_{T\to\infty} |\inner{h}{\hat{\pi}^x_T} - \inner{h}{\hat{\pi}^{x_0}_T}| = 0$ and hence the result follows since \eqref{eq: p-ergodic_A_all_x} holds for $x_0$.

It remains to prove that $\qprob\bra{\int_0^\infty \Delta_t \ud t < \infty} = 1$.  By way of contradiction assume there is some $0<\delta\leq 1$ so that $\qprob\bra{\int_0^\infty \Delta_t \ud t = \infty} = \delta$.  Then, for each $N$ it holds that $\qprob\bra{\int_0^\infty \Delta_t \ud t > N}\geq \delta$, which in turn implies $\lim_{T\to\infty} \qprob\bra{\int_0^T \Delta_t \ud t > N} \geq \delta$. By construction, for any fixed $T>0$ the law of $\Delta$ on $[0,T]$ under $\qprob$ coincides with the law of $D$ under $\prob$ on $[0,T]$.  It this holds that $\lim_{T\to\infty} \prob\bra{\int_0^T D_t \ud t > N} \geq \delta$.  But, this gives $\prob\bra{\int_0^\infty D_t \ud t > N} \geq \delta$ for all $N$ and hence $\prob\bra{\int_0^\infty D_t \ud t = \infty} > 0$.  But this violates Assumptions \ref{ass: standing} since $\lim_{t\to\infty} e^{\kappa t}D_t = 0$ $\prob$ almost surely for some $\kappa >0$. Thus, $\qprob\bra{\int_0^\infty \Delta_t \ud t < \infty} = 1$ finishing the proof.
\end{proof}

With all the above lemmas, the proof of Theorem \ref{thm:main} is now given.

\begin{proof}[Proof of Theorem \ref{thm:main}]

We again adopt the notation in \eqref{eq: pi_int}. In view of Lemma \ref{lem: horizon_T_not_nec} the remaining statement Theorem \ref{thm:main} which must be proved is that there is a set $\Omega_0\in\F_{\infty}$ with $\qprob\bra{\Omega_0} = 1$ such that \eqref{eq: main_pi_conv} holds: i.e.
\begin{equation*}
\omega\in \Omega_0 \Longrightarrow \lim_{T\to\infty} \inner{h}{\hat{\pi}^x_T}(\omega) = \inner{h}{\pi}\ \textrm{for all}\ x>0, h\in C_b(F;\Real).
\end{equation*}
Recall the definition of $\mathcal{C}$ from \eqref{eq: lip-est_A} and let $h\in C_b(F;\Real)\cap \mathcal{C}$. In view of Lemma \ref{lem: ergodic_deg_all_x} there is a set $\Omega_h\in\F_{\infty}$ such that $\qprob\bra{\Omega_h} = 1$ and
\begin{equation*}
\omega\in \Omega_h \Longrightarrow \lim_{T\to\infty} \inner{h}{\hat{\pi}^x_T}(\omega) = \inner{h}{\pi}\ \textrm{for all}\ x>0.
\end{equation*}

Let the (countable subset) $\tilde{\mathcal{C}}\subset \mathcal{C}$ be as in the technical Lemma \ref{lem: approximation} below and set $\Omega_0 = \cap_{h\in\tilde{\mathcal{C}}} \Omega_h$. Clearly, $\qprob\bra{\Omega_0} = 1$.  Let $\omega\in\Omega_0$ and $h\in C_b(F;\Real)$ with $C=\sup_{y\in F}|h(y)|$. Let $\eps > 0$ and for $n\geq 5$ take $^\uparrow \phi^{n}_{m,k}, ^\downarrow \phi^{n}_{m,k}$ and $\theta^n$ as in Lemma \ref{lem: approximation} such that \eqref{eq: approx_values} therein holds.  In what follows the $\omega$ will be suppressed, but all evaluations are understood to hold for this $\omega$.

Let $x>0$. With $\nu$ from \eqref{eq: approx_values} equal to $\hat{\pi}^x_T$ it follows that
\begin{equation*}
\inner{^\uparrow \phi^n_{m,k}}{\hat{\pi}^x_T} - 2C\inner{1-\theta^{n-4}}{\hat{\pi}^x_T} - 2\eps \leq \inner{h}{\hat{\pi}^x_T} \leq \inner{^\downarrow \phi^n_{m,k}}{\hat{\pi}^x_T} + 2C\inner{1-\theta^{n-4}}{\hat{\pi}^x_T} + 2\eps.
\end{equation*}
With $\nu$ from \eqref{eq: approx_values} equal to $\pi$ one obtains
\begin{equation*}
\inner{^\uparrow \phi^n_{m,k}}{\pi} - 2C\inner{1-\theta^{n-4}}{\pi} - 2\eps \leq \inner{h}{\pi} \leq \inner{^\downarrow \phi^n_{m,k}}{\pi} + 2C\inner{1-\theta^{n-4}}{\pi} + 2\eps.
\end{equation*}
Putting these two together yields
\begin{equation*}
\begin{split}
\inner{h}{\hat{\pi}^x_T} - \inner{h}{\pi} &\geq \inner{^\uparrow \phi^n_{m,k}}{\hat{\pi}^x_T} - 2C\inner{1-\theta^{n-4}}{\hat{\pi}^x_T} - 2\eps - \left(\inner{^\downarrow \phi^n_{m,k}}{\pi} + 2C\inner{1-\theta^{n-4}}{\pi} + 2\eps\right)\\
&= \inner{^\uparrow \phi^n_{m,k}}{\hat{\pi}^x_T} - \inner{^\downarrow \phi^n_{m,k}}{\pi} - 2C\left(\inner{1-\theta^{n-4}}{\hat{\pi}^x_T} + \inner{1-\theta^{n-4}}{\pi}\right) - 4\eps.
\end{split}
\end{equation*}
Since $\theta^{n-4},^\uparrow\phi^n_{m,k}, ^\downarrow \phi^n_{m,k} \in \tilde{\mathcal{C}}\subset\mathcal{C}$ taking $T\to\infty$ gives
\begin{equation*}
\liminf_{T\to\infty} \inner{h}{\hat{\pi}^x_T} - \inner{h}{\pi} \geq \inner{^\uparrow \phi^n_{m,k}}{\pi} - \inner{^\downarrow \phi^n_{m,k}}{\pi} - 4C\inner{1-\theta^{n-4}}{\pi} - 4\eps.
\end{equation*}
Now, from Lemma \ref{lem: approximation} we know for fixed $m,n$ that the functions $^\uparrow\phi^{n}_{m,k}$ and $^\downarrow \phi^n_{m,k}$ are increasing and decreasing respectively in $k$ and such that a) $\lim_{k\to\infty} ^\downarrow\phi^n_{m,k}(y) - ^\uparrow \phi^n_{m,k}(y) = 0$ for $y\in \bar{F}_{n-2}$ and b) $|^\uparrow\phi^n_{m,k}(y)-^\uparrow\phi^n_{m,k}(y)| \leq 2C + 2\eps$ for all $y\in F$ and $n,m,k$. Therefore, taking $k\to\infty$ in the above and using the monotone convergence theorem we obtain
\begin{equation*}
\liminf_{T\to\infty} \inner{h}{\hat{\pi}^x_T} - \inner{h}{\pi} \geq -2(C+\eps)\pi\bra{\bar{F}_{n-2}^c} - 4C\inner{1-\theta^{n-4}}{\pi} - 4\eps.
\end{equation*}
From Lemma \ref{lem: approximation} we know that $0\leq \theta^n(y)\leq 1$, $\lim_{n\to\infty}\theta^n(y) = 1$ for all $y\in F$. Thus, by the bounded convergence theorem and the fact that $\pi$ is tight in $F$ it follows that by taking $n\uparrow\infty$:
\begin{equation*}
\liminf_{T\to\infty} \inner{h}{\hat{\pi}^x_T} - \inner{h}{\pi} \geq - 4\eps.
\end{equation*}
Taking $\eps\downarrow 0$ gives that $\liminf_{T\to\infty} \inner{h}{\hat{\pi}^x_T} - \inner{h}{\pi} \geq 0$.  Thus, we have just shown for $\omega\in \Omega_0$, $x>0$ and $h\in C_b(F;\Real)$ that
\begin{equation*}
\liminf_{T\to\infty} \inner{h}{\hat{\pi}^x_T}(\omega) - \inner{h}{\pi}\geq 0.
\end{equation*}
By applying the above to $\hat{h} = -h \in C_b(F;\Real)$ we see that
\begin{equation*}
\limsup_{T\to\infty} \inner{h}{\hat{\pi}^x_T}(\omega) - \inner{h}{\pi} \leq 0,
\end{equation*}
which finishes the proof.
\end{proof}

%---comment out this stuff ----_%

%\subsection{Main Results : Degenerate $\eta$}\label{subsubsec: eta_deg}

%We now consider when $\eta$ is allowed to degenerate on $E$ Here, an approximation argument is used to prove type of convergence similar to that in \eqref{eq: p-ergodic_AA}. There are, however, two key differences: first, the probability measure under which convergence takes place is $\probt^{x}$ where $\chi_0 = x > 0$ is arbitrary but $\ze_0\sim p$. Secondly, \eqref{eq: p-ergodic_AA} is shown to hold for all continuous bounded functions, rather than all bounded measurable functions.  However, construction of a set similar to $\widetilde{\Omega}^{z,x}$ in Theorem \ref{thm: backwards_ergodicity} will still take place.

%\begin{thm}\label{thm: conv_for_chi}
%Let Assumptions \ref{ass: standing} hold. Let $x>0$. Then there exists a set $\widetilde{\Omega}^x\subseteq \widetilde{\Omega}$ with $\probt^x\bra{\widetilde{\Omega}^x} = 1$ such that for each $\tilde{\omega}\in\widetilde{\Omega}^x$ it holds that $L_T(\tilde{\omega})\rightarrow \pi$ weakly.
%\end{thm}

%---end of commenting ot ----%

%%%%%%%%%%%  APPENDIX %%%%%%%%%%%%%%%%%%%%%%%%%

\begin{appendix}

\section{Some Technical Results}

\begin{lem}\label{lem: tilde_LR_calc}

Let Assumptions \ref{ass: standing} hold, and additionally assume that $|\eta| > 0$, $f\in C^2(E;\Real_+)$.  Recall $F$ from \eqref{eq: cond_cdf} and the invariant density $p$ for $Z$. Let $h\in C^2(F)$ be given and set \begin{equation}\label{eq: phi_psi_def}
\phi(z,x) \dfn p(z)h(z,x);\qquad \psi(z,x)\dfn \int_0^x h(z,y)dy.
\end{equation}
Let the operator $L$ be as in \eqref{eq: L_def} and the operator $L^R = A^{ij}\partial^2_{ij} + (b^R)^i\partial_i$ be as in the proof of Proposition \ref{prop: backwards_ergodicity}, where $A$ is from \eqref{eq: bA_def} and $b^R$ is from \eqref{eq: b_R_def}.  Let $\tilde{L}^R$ be the formal adjoint of $L^R$. Then $\tilde{L}^R \phi = p\ \partial_{x}\left(L\psi\right)$. In particular, if $L\psi= 0$ then $\tilde{L}^R\phi = 0$.
\end{lem}

\begin{proof}

For notational ease, the arguments will be suppressed when writing functions except for the $x$ appearing in the drifts and volatilities of the operators. Now, recall the dynamics for the reversed process $(\ze,\chi)$ in \eqref{eq: ze_chi_dyn_noT}:
\begin{equation*}
\begin{split}
d\ze_t &= \left(c\frac{\nabla p}{p} + \dvg{c} - m\right)(\ze_t)dt + \sigma(\ze_t)dW_t\\
d\chi_t&= \left(f(\ze_t) - \chi_t\left(a - \theta'c\frac{\nabla p}{p} - \nabla\cdot(c\theta)\right)(\ze_t)\right)dt + \chi_t\left(\theta'c(\ze_t)dW_t + \eta(\ze_t)'dB_t\right),
\end{split}
\end{equation*}
and note, as is mentioned in the proof of Proposition \ref{prop: backwards_ergodicity}, that $L^R$ is the generator for $(\ze,\chi)$. To further simplify the calculations, set
\begin{equation}\label{eq: xi_def}
\xi  \dfn  \frac{1}{2}\left(c\frac{\nabla p}{p} + \dvg{c}\right) -m,
\end{equation}
and
\begin{equation}\label{eq: F_funct_def}
H(c,\theta) \dfn \nabla\cdot(c\theta) - \theta'\dvg{c}.
\end{equation}
Note that by \eqref{eq: p_density_pde} it follows that $0 = \nabla\cdot (p\xi)$. With this notation we have that
\begin{equation*}
\begin{split}
d\ze_t &= \left(m + 2\xi\right)(\ze_t)dt + \sigma(\ze_t)dW_t\\
d\chi_t&= \left(f(\ze_t) - \chi_t\left(a - 2\theta'(m+\xi) - H(c,\theta)\right)(\ze_t)\right)dt + \chi_t\left(\theta'c(\ze_t)dW_t + \eta(\ze_t)'dB_t\right),
\end{split}
\end{equation*}
which in turns yields that
\begin{equation}\label{eq: temp_A_b_R}
A = \left(\begin{array}{c c} c & xc\theta\\ x\theta'c & x^2(\theta'c\theta + \eta'\eta)\end{array}\right);\qquad b^R = \left(\begin{array}{c} m + 2\xi \\ f - x\left(a - 2\theta'(m+\xi) - H(c,\theta)\right)\end{array}\right).
\end{equation}
along with
\begin{equation}\label{eq: temp_b}
b = \left(\begin{array}{c} m\\ -f +x\left(a + \theta'c\theta + \eta'\eta\right)\end{array}\right).
\end{equation}
Lastly, multivariate notation will be used for derivatives with respect to $z$ and single variate
notation used for derivatives with respect to $x$.  Thus, for the given $\phi$:
\begin{equation*}
\nabla_{(z,x)} \phi = (\nabla \phi, \dot{\phi});\qquad D^2_{(z,x)}\phi = \left(\begin{array}{c
      c} D^2 \phi & \nabla (\dot{\phi}) \\ \nabla (\dot{\phi})' &
    \ddot{\phi}\end{array}\right).
\end{equation*}
Since $\phi = p\ h$ and $p$ is not a function of $x$:
\begin{equation*}
\nabla_{(z,x)}\phi = \left(\begin{array}{c} p\nabla h + h\nabla p \\
    p\dot{h}\end{array}\right).
\end{equation*}
By definition, $\tilde{L}_R\phi = \nabla_{(z,x)}\cdot\left((1/2)(A\nabla_{(z,x)}\phi +
  \phi \dvgalt{A}) - b^R\phi\right)$. Using \eqref{eq: temp_A_b_R}:
\begin{equation*}
A\nabla_{(z,x)}\phi =
\left(\begin{array}{c} pc\nabla h + hc\nabla p + px\dot{h}c\theta \\
    px\theta'c\nabla h  + hx\theta'c\nabla p +
    px^2\dot{h}(\theta'c\theta + \eta'\eta)\end{array}\right).
\end{equation*}
Calculation shows
\begin{equation*}
\dvgalt{A} = \left(\begin{array}{c} \dvg{c} + c\theta \\
    x\nabla\cdot(c\theta) +2x(\theta'c\theta + \eta'\eta)\end{array}\right),
\end{equation*}
so that
\begin{equation*}
\begin{split}
\frac{1}{2}&(A\nabla_{(z,x)} \phi + \phi \dvgalt{A})\\
&=\frac{1}{2}\left(\begin{array}{c} pc\nabla h + hc\nabla p +
    px\dot{h}c\theta + ph\dvg{c} + phc\theta \\ px\theta'c\nabla h  + hx\theta'c\nabla p +
    px^2\dot{h}(\theta'c\theta + \eta'\eta) + pxh\nabla\cdot(c\theta)
    +2pxh(\theta'c\theta + \eta'\eta)\end{array}\right).
\end{split}
\end{equation*}
This gives $(1/2)(A\nabla_{(z,x)} \phi + \phi \dvgalt{A}) - b^R\phi = (\mathbf{A},\mathbf{B})'$
where
\begin{equation}\label{eq: temp_boldAB_val}
\begin{split}
\mathbf{A} & = \frac{1}{2}\left(pc\nabla h + hc\nabla p +
    px\dot{h}c\theta + ph\dvg{c} + phc\theta\right) - phm - 2ph\xi,\\
\mathbf{B} & = \frac{1}{2}\left(px\theta'c\nabla h  + hx\theta'c\nabla p +
    px^2\dot{h}(\theta'c\theta + \eta'\eta) + pxh\nabla\cdot(c\theta)
    +2pxh(\theta'c\theta + \eta'\eta)\right)\\
&\qquad\qquad -phf + pxha - 2pxh\theta'(m+\xi) - pxhH(c,\theta).
\end{split}
\end{equation}
Now, $\tilde{L}_R\phi = \nabla\cdot\mathbf{A} + \dot{\mathbf{B}}$. $\mathbf{A}$ is
treated first. From \eqref{eq: xi_def} it follows that $p\dvg{c} + c\nabla
p = 2p(m +\xi)$ and hence
\begin{equation*}
2\mathbf{A} = pc\nabla h + px\dot{h}c\theta +
phc\theta - 2ph\xi.
\end{equation*}
For a scalar function $f$ and $\Real^d$ valued function $g$, $\nabla\cdot(fg)
= f\nabla\cdot g + \nabla f'g$. Using this
\begin{equation*}
\begin{split}
2\nabla\cdot \mathbf{A} &= p\nabla\cdot(c\nabla h) + \nabla
h'c\nabla p + px\dot{h}\nabla\cdot(c\theta) + x\nabla(p\dot{h})'c\theta + ph\nabla\cdot(c\theta)\\
&\qquad\qquad + \nabla(ph)'c\theta -
2h\nabla\cdot(p\xi) - 2p\nabla h'\xi,\\
&= p\nabla\cdot(c\nabla h) + \nabla
h'c\nabla p + px\dot{h}\nabla\cdot(c\theta) + px\nabla(\dot{h})'c\theta +
x\dot{h}\nabla p'c\theta + ph\nabla\cdot(c\theta)\\
&\qquad\qquad + p\nabla h'c\theta +h\nabla p'c\theta-
2h\nabla\cdot(p\xi) - 2p\nabla h'\xi.\\
\end{split}
\end{equation*}
Using that $\nabla\cdot(c\nabla h) = \textrm{tr}\left(cD^2h\right) + \nabla
h'\dvg{c}$ and collecting terms by derivatives of $h$ gives
\begin{equation*}
\begin{split}
2\nabla\cdot \mathbf{A} &=p\textrm{tr}\left(cD^2h\right) +
px\nabla(\dot{h})'c\theta + \nabla
h'\left(p\dvg{c} + c\nabla p + pc\theta -
  2p\xi\right),\\
&\qquad + \dot{h}\left(px\nabla\cdot(c\theta) +
  x\nabla p'c\theta\right) +
h\left(p\nabla\cdot(c\theta) + \nabla p'c\theta - 2\nabla\cdot(p\xi)\right).
\end{split}
\end{equation*}
Since $p\dvg{c}+c\nabla p = 2p(m + \xi)$, $\nabla\cdot(p\xi)= 0$ and
$\nabla\cdot(c\theta) = H(c,\theta) + \theta'\dvg{c}$,
\begin{equation*}
\begin{split}
p\dvg{c} + c\nabla p + pc\theta -
  2p\xi&=2pm + pc\theta,\\
px\nabla\cdot(c\theta) +
  x\nabla p'c\theta&=2px\theta'(m+\xi) + pxH(c,\theta),\\
p\nabla\cdot(c\theta) + \nabla p'c\theta - 2\nabla\cdot(p\xi)&=2p\theta'(m+\xi)+pH(c,\theta).
\end{split}
\end{equation*}
Plugging this in and factoring out the $p$ yields
\begin{equation}\label{eq: temp_boldA_val}
\begin{split}
\frac{2}{p}\nabla\cdot\mathbf{A} &= \textrm{tr}\left(cD^2h\right) +
x\nabla(\dot{h})'c\theta +\nabla
h'\left(2m+c\theta\right) + \dot{h}\left(2x\theta'(m+\xi) +
  xH(c,\theta)\right)\\
&\qquad\qquad + h\left(2\theta'(m+\xi) +
  H(c,\theta)\right).
\end{split}
\end{equation}
Turning to $\mathbf{B}$ in \eqref{eq: temp_boldAB_val}. Using $p\dvg{c} + c\nabla p =
2p(m+\xi)$ and $\nabla\cdot(c\theta) = H(c,\theta) + \theta'\dvg{c}$ yields
\begin{equation*}
2\mathbf{B}=px\theta'c\nabla h - 2pxh\theta'(m +\xi) +
    px^2\dot{h}(\theta'c\theta + \eta'\eta)+2pxh(\theta'c\theta + \eta'\eta)-2phf + 2pxha  -pxhH(c,\theta).
\end{equation*}
Since only $h$ depends upon $x$,
\begin{align*}
2\dot{\mathbf{B}} & = p\theta'c\nabla h + px\nabla(\dot{h})'c\theta -
2ph\theta'(m+\xi) - 2px\dot{h}\theta'(m+\xi) +
2px\dot{h}(\theta'c\theta+\eta'\eta) +
px^2\ddot{h}(\theta'c\theta+\eta'\eta)\\
&+ 2ph(\theta'c\theta + \eta'\eta) + 2px\dot{h}(\theta'c\theta + \eta'\eta)
-2p\dot{h}f + 2pha + 2px\dot{h}a-phH(c,\theta)-px\dot{h}H(c,\theta).
\end{align*}
Grouping terms by derivatives of $h$ and factoring out the $p$ yields
\begin{align}
\frac{2}{p}\dot{\mathbf{B}} & = x\ddot{h}(\theta'c\theta+\eta'\eta) +
x\nabla(\dot{h})'c\theta +
h\left(-2\theta'(m+\xi) + 2(\theta'c\theta + \eta'\eta) + 2a -hH(c,\theta)\right)+ \nabla h'c\theta  \label{eq: temp_boldB_val}\\
\nonumber &+ \dot{h}\left(-2x\theta'(m+\xi) +
  4x(\theta'c\theta+\eta'\eta)-2f+2xa-xH(c,\theta)\right).
\end{align}
Putting together \eqref{eq: temp_boldA_val} and \eqref{eq: temp_boldB_val} and
using that $\tilde{L}_R\phi = \nabla\cdot\mathbf{A}+\dot{\mathbf{B}}$:
\begin{equation}\label{eq: tildeL_R_Val}
\begin{split}
\frac{1}{p}\tilde{L}_R\phi & =
\frac{1}{2}\textrm{tr}\left(cD^2h\right) + x\nabla(\dot{h})'c\theta +
\frac{1}{2}x^2\ddot{h}(\theta'c\theta + \eta'\eta) + \nabla
h'(m+c\theta)\\
&+ \dot{h}\left(2x(\theta'c\theta + \eta'\eta) - f + xa\right) +
h\left(\theta'c\theta + \eta'\eta + a\right).
\end{split}
\end{equation}
Turning now to $\psi$, since
\begin{equation*}
\begin{split}
L\psi = \frac{1}{2}\textrm{tr}\left(cD^2\psi\right) +
x\nabla(\dot{\psi})'c\theta +
\frac{1}{2}x^2\ddot{\psi}(\theta'c\theta+\eta'\eta) + \nabla\psi'm  +
\dot{\psi}\left(-f + xa + x(\theta'c\theta + \eta'\eta)\right),
\end{split}
\end{equation*}
it follows that (note : only $\psi$ depends upon $x$ and $\dot{\psi} = h$)
\begin{equation*}
\begin{split}
\dot{L\psi} &= \frac{1}{2}\textrm{tr}\left(cD^2\dot{\psi}\right) +
x\nabla(\ddot{\psi})'c\theta + \nabla(\dot{\psi})'c\theta+
x\ddot{\psi}(\theta'c\theta+\eta'\eta) +
\frac{1}{2}x^2\dddot{\psi}(\theta'c\theta+\eta'\eta)\\
&\qquad  + \nabla(\dot{\psi})'m  +
\ddot{\psi}\left(-f + xa + x(\theta'c\theta + \eta'\eta)\right) +
\dot{\psi}\left(a + \theta'c\theta + \eta'\eta\right),\\
&= \frac{1}{2}\textrm{tr}\left(cD^2\dot{\psi}\right) +
x\nabla(\ddot{\psi})'c\theta  +
\frac{1}{2}x^2\dddot{\psi}(\theta'c\theta+\eta'\eta)\\
&\qquad  + \nabla(\dot{\psi})'(m+c\theta)  +
\ddot{\psi}\left(2x(\theta'c\theta+\eta'\eta) -f + xa\right) +
\dot{\psi}\left(a + \theta'c\theta + \eta'\eta\right),\\
&= \frac{1}{2}\textrm{tr}\left(cD^2h\right) +
x\nabla(\dot{h})'c\theta  +
\frac{1}{2}x^2\ddot{h}(\theta'c\theta+\eta'\eta)\\
&\qquad  + \nabla h'(m+c\theta)  +
\dot{h}\left(2x(\theta'c\theta+\eta'\eta) -f + xa\right) +
h\left(a + \theta'c\theta + \eta'\eta\right).
\end{split}
\end{equation*}
But, from \eqref{eq: tildeL_R_Val} this last term is precisely
$(1/p)\tilde{L}_R\phi$.
\end{proof}

%---------------------------------------------%

\begin{lem}\label{lem: approximation}
Let Assumption \ref{ass: standing} hold. Let $\mathcal{C}$ be as in \eqref{eq: lip-est_A}. Recall that $F=E\times(0,\infty)$ and let $\cbra{F_n}_{n\in\mathbb{N}}$ be a family of open, bounded, increasing subsets of $F$ with smooth boundary such that $F=\cup_n F_n$.  There exists a countable family of functions
\begin{equation}\label{eq: function_family}
\tilde{\mathcal{C}}\dfn \cbra{^\uparrow \phi^n_{m,k}, ^\downarrow \phi^n_{m,k}, \theta^n\ \such \ n,m,k\in\mathbb{N}, n\geq 3}\subset \mathcal{C}
\end{equation}
such that
\begin{enumerate}[1)]
\item For each $n\geq 3$, $0\leq \theta^n\leq 1$ with $\theta^n = 1$ on $\bar{F}_n$ and $\theta^n = 0$ on $F^c_{n+1}$.
\item For each $n\geq 3$ and $m$, the functions $^\uparrow \phi^n_{m,k}$ are increasing in $k$ and the functions $^\downarrow \phi^n_{m,k}$ are decreasing in $k$. Furthermore, for any $n\geq 3$ and $m$, $\lim_{k\to\infty} |^\uparrow \phi^n_{m,k}(y) - ^\downarrow \phi^n_{m,k}(y)| = 0$ for $y\in \bar{F}_{n-2}$.
\end{enumerate}
Additionally, for any $h\in C_b(F;\Real)$ set $C = C(h) \dfn \sup_{y\in F}|h(y)|$. Then, for any $\eps > 0$ and any integer $n\geq 5$ there exits an integer $m = m(\eps,n)$ such that for all $k\in \mathbb{N}$, $\sup_{y\in F}|^\uparrow\phi^n_{m,k}(y)|\leq C+\eps$, $\sup_{y\in F}|^\downarrow \phi^n_{m,k}(y)|\leq C+\eps$. Furthermore, for any Borel measure $\nu$ on $F$:
\begin{equation}\label{eq: approx_values}
\begin{split}
\int_F\ ^\uparrow \phi^n_{m,k}d\nu - 2C\int_F (1-\theta^{n-4})d\nu - 2\eps \leq \int_F hd\nu  \leq \int_F \ ^\downarrow \phi^n_{m,k}d\nu + 2C\int_F (1-\theta^{n-4})d\nu + 2\eps.
\end{split}
\end{equation}
\end{lem}

\begin{proof}[Proof of Lemma \ref{lem: approximation}]
Fix $n\in\mathbb{N}$ and let $(\phi^n_m)_{m\in\mathbb{M}}$ be a countable dense (with respect to the supremum norm) subset of $C_b(\bar{F}_n;\Real)$.  Now, let $k\in\mathbb{N}$ and define:
\begin{equation}\label{eq: tilde_phi_nmk}
^\uparrow \tilde{\phi}^{n}_{m,k}(y) \dfn \inf_{y_0\in \bar{F}_n}\pare{\phi^n_m(y_0) + k|y-y_0|};\quad ^\downarrow \tilde{\phi}^{n}_{m,k}(y) \dfn \sup_{y_0\in \bar{F}_n}\pare{\phi^n_m(y_0) - k|y-y_0|};\qquad y\in \bar{F}_n.
\end{equation}
As shown in \cite[Ch. 3.4]{MR2378491}, $^\uparrow \tilde{\phi}^{n}_{m,k}$ and $^\downarrow \tilde{\phi}^n_{m,k}$ are a) increasing and decreasing respectively in $k$, and b) Lipschitz continuous in $\bar{F}_n$ with Lipschitz constant $k$. Furthermore, as $k\uparrow\infty$, $^{\uparrow}\tilde{\phi}^{n}_{m,k}\nearrow\phi^n_m$ and $^{\downarrow}\phi^{n}_{m,k}\searrow \phi^{n}_m$ on $\bar{F}_n$.

Next, let $\theta^n\in C^\infty(F;\Real)$ be such that $0\leq \theta^n\leq 1$, $\theta^n(y) = 1$ on $\bar{F}_n$ and $\theta^n(y) = 0$ on $F_{n+1}^c$. Clearly, $\theta^n\in\mathcal{C}$ for each $n$. Now, assume $n\geq 3$ and extend $^\uparrow\tilde{\phi}^n_{m,k}$ and $^\downarrow \tilde{\phi}^n_{m,k}$ from functions on $\bar{F}_n$ to all of $F$ via
\begin{equation*}
^\uparrow \phi^n_{m,k}(y) = \begin{cases} ^\uparrow\tilde{\phi}^n_{m,k}(y)\theta^{n-2}(y) & y\in\bar{F}_n\\ 0 & else\end{cases};\qquad ^\downarrow \phi^n_{m,k}(y) = \begin{cases} ^\downarrow\tilde{\phi}^n_{m,k}(y)\theta^{n-2}(y) & y\in\bar{F}_n\\ 0 & else\end{cases}
\end{equation*}
Clearly, $^{\uparrow}\phi^n_{m,k}$ and $^\downarrow\phi^n_{m,k}$ are Lipschitz on $F$ and, since $F_n$ is bounded, it also holds that $^\uparrow \phi^n_{m,k}$ and $^\downarrow \phi^n_{m,k}$ are in $\mathcal{C}$. Note also that $^\uparrow\phi^n_{m,k}$ and $^\downarrow \phi^n_{m,k}$ increase and decrease respectively as $k\uparrow\infty$ to a function which is equal to  $\phi^n_m$ on $\bar{F}_{n-2}$ and that $^\uparrow\phi^n_{m,k}, ^\downarrow\phi^n_{m,k}$ are bounded on all of $F$ by $\sup_{y\in\bar{F}_n}|^\uparrow\tilde{\phi}^n_{m,k}(y)|$ and $\sup_{y\in\bar{F}_n}|^\downarrow\tilde{\phi}^n_{m.k}(y)|$ respectively. This proves $1),2)$ above.

Now, let $h\in C_b(F;\Real)$ with $C = \sup_{y\in F}|h(y)|$. Let $\eps > 0$ and for $n\geq 5$ choose $m=m(\eps,n)$ so that $\sup_{y\in\bar{F}_n}|h(y)-\phi^n_m(y)| \leq \eps$. By construction of $^\uparrow \tilde{\phi}^n_{m,k}$ in \eqref{eq: tilde_phi_nmk} it follows for each $k$ that
\begin{equation*}
-(C+\eps) \leq \inf_{y_0\in\bar{F}_n}(\phi^n_m(y_0)) \leq\ ^\uparrow\tilde{\phi}^n_{m,k}(y) \leq \phi^n_m(y) \leq h(y) + \eps\leq C+\eps;\qquad y \in \bar{F}_n.
\end{equation*}
By definition of $^\uparrow \phi^n_{m,k}$ this gives $\sup_{y\in F}|^\uparrow \phi^n_{m,k}(y)|\leq C+\eps$. Furthermore, since $\theta^{n-2}(y) = 1$ on $\bar{F}_{n-2}$, we have $h(y) \geq ^\uparrow \phi^n_{m,k}(y) - \eps$ on $\bar{F}_{n-2}$. Therefore, for any Borel measure $\nu$, using the notation in \eqref{eq: pi_int}:
\begin{equation*}
\begin{split}
\inner{h}{\nu} &\geq \inner{\left(^\uparrow\phi^n_{m,k}-\eps\right)1_{\bar{F}_{n-2}}}{\nu} - C\nu\bra{\bar{F}_{n-2}^c}\\
&\geq \inner{^\uparrow\phi^n_{m,k}}{\nu} - \inner{^\uparrow\phi^{n}_{m,k}1_{\bar{F}_{n-2}^c}}{\nu} - \eps - C\nu\bra{\bar{F}_{n-2}^c}\\
&\geq \inner{^\uparrow\phi^n_{m,k}}{\nu} - (C+\eps)\nu\bra{\bar{F}_{n-2}^c} - \eps - C\nu\bra{\bar{F}_{n-2}^c}\\
&\geq \inner{^\uparrow\phi^n_{m,k}}{\nu} - 2C\nu\bra{\bar{F}_{n-2}^c} - 2\eps\\
&\geq \inner{^\uparrow\phi^n_{m,k}}{\nu} - 2C\int_F (1-\theta^{n-4})d\nu - 2\eps,\\
\end{split}
\end{equation*}
where the last inequality follows since $1_{\bar{F}^c_{n-2}(y)}\leq 1-\theta^{n-4}(y)$. This gives the lower bound in \eqref{eq: approx_values}. A similar calculation shows for all $k$ that
\begin{equation*}
\begin{split}
-(C+\eps) \leq h(y)-\eps \leq \phi^n_m(y) \leq ^\downarrow \tilde{\phi}^n_{m,k}(y) \leq \sup_{y_0\in\bar{F}_n}(\phi^n_m(y_0)) \leq C + \eps;\qquad y\in \bar{F}_n.
\end{split}
\end{equation*}
This gives $\sup_{y\in F}|^\downarrow \phi^n_{m,k}(y)|\leq C+\eps$ and $h(y) \leq ^\downarrow \phi^n_{m,k}(y) +\eps$ on $\bar{F}_{n-2}$. Thus
\begin{equation*}
\begin{split}
\inner{h}{\nu} &\leq \inner{\left(^\downarrow\phi^n_{m,k}+\eps\right)1_{\bar{F}_{n-2}}}{\nu} + C\nu\bra{\bar{F}_{n-2}^c}\\
&\leq \inner{^\downarrow\phi^n_{m,k}}{\nu} - \inner{^\downarrow\phi^{n}_{m,k}1_{\bar{F}_{n-2}^c}}{\nu} + \eps + C\nu\bra{\bar{F}_{n-2}^c}\\
&\leq \inner{^\downarrow\phi^n_{m,k}}{\nu} + (C+\eps)\nu\bra{\bar{F}_{n-2}^c} + \eps + C\nu\bra{\bar{F}_{n-2}^c}\\
&\leq \inner{^\downarrow\phi^n_{m,k}}{\nu} + 2C\nu\bra{\bar{F}_{n-2}^c} + 2\eps\\
&\leq \inner{^\downarrow\phi^n_{m,k}}{\nu} + 2C\int_F (1-\theta^{n-4})d\nu + 2\eps.\\
\end{split}
\end{equation*}
Therefore, the upper bound in \eqref{eq: approx_values} is established.
\end{proof}

\end{appendix}

\bibliographystyle{siam}
\bibliography{master}
\end{document}